\documentclass[11pt,a4paper]{article}

\usepackage[T1]{fontenc}
\usepackage[utf8]{inputenc}
\usepackage[margin=1in]{geometry}
\usepackage{amsmath,amssymb,amsfonts,amsthm}
\usepackage{graphicx}
\usepackage{array,booktabs,multirow}
\usepackage{subfig}
\usepackage{caption}
\usepackage{algorithm}
\usepackage{algorithmic}
\usepackage{placeins}
\usepackage{xcolor}
\usepackage[hidelinks]{hyperref}
\hypersetup{
  pdftitle={A Structure-Preserving Scheme for Computing the Gradient of the Eikonal Equation},
  pdfauthor={Ka Lun Leung and Shingyu Leung}
}

\newtheorem{proposition}{Proposition}

\newcommand{\R}{\mathbb{R}}

\newcommand{\grad}{\nabla}

\title{A Structure-Preserving Scheme for Computing the Gradient of the Eikonal Equation}
\author{
  Ka Lun Leung \qquad Shingyu Leung\\[0.5em]
  {\small Department of Mathematics}\\
  {\small The Hong Kong University of Science and Technology}\\
  {\small Clear Water Bay, Hong Kong}\\[0.3em]
  {\small\texttt{klleungat@connect.ust.hk}, \texttt{masyleung@ust.hk}}
}
\date{}

\begin{document}
\maketitle

\begin{abstract}
{We propose a first-order, structure-preserving method for directly computing the normalized traveltime gradient of the eikonal equation in smooth first-arrival regions. The method combines analytical derivatives of spherical linear interpolation (SLERP) in a Godunov-type upwind discretization with spherical forward Euler (SFE) pseudo-time integration. }{This construction gives a tangent residual and keeps every iterate on $\mathbb{S}^2$ without projection or normalization. }{The SLERP discretization exactly preserves a family of curved orientation equilibria. }{Removing the background discretization bias improves absolute orientation accuracy near these equilibria, allowing a prescribed error tolerance to be reached on a coarser grid. }{We prove the equilibrium-preservation property and, separately, a smaller angular error for SFE than normalized Euler in a one-cell relaxation under the stated step-size condition. }{Two- and three-dimensional experiments examine convergence, geometric preservation, and computational cost. }{Comparisons with gradient recovery from scalar traveltimes show more accurate directions on the same grid in the curved media considered. The method provides direction fields for wave propagation and geometric optics, with potential use in adjoint-based inverse problems.}
\end{abstract}

\section{Introduction} 
\label{Sec:Introduction}

{In this work, we develop a structure-preserving discretization for directly computing the direction of the traveltime gradient associated with the eikonal equation}
\begin{equation} \label{eq:Eikonal_equation}
    \|\nabla \phi(\mathbf{x})\| = \frac{1}{c(\mathbf{x})}, 
    \qquad c(\mathbf{x}) > 0,
\end{equation}
where $\phi(\mathbf{x})$ denotes the scalar potential function and $c(\mathbf{x})$ the spatially varying wave speed.  
Of particular interest is the gradient field $\nabla \phi$, which characterizes the local propagation direction of the wavefront.  
{To compute this direction directly, we construct a Godunov-type upwind discretization using analytical endpoint derivatives of SLERP and evolve the resulting system by the spherical forward Euler (SFE) method. The scheme gives a first-order approximation of smooth orientation fields and preserves the unit-sphere constraint at every pseudo-time step.}

{The proposed spatial discretization exactly preserves a family of curved orientation equilibria, including those rotated relative to the grid. For these fields, the discrete directional turning balances the refraction due to the medium. This balance removes the background discretization bias and improves the absolute accuracy of the orientation near the equilibria. We also examine the effect of time integration using a one-cell relaxation, for which SFE gives a smaller angular error under a suitable step-size condition. The spatial discretization and pseudo-time update are compared separately in the numerical experiments, where we study the accuracy of the computed direction fields and the cost of reaching prescribed error tolerances.}

A representative example arises from the scalar Helmholtz equation describing a time-harmonic wavefield $U(z,x,\omega)$ in an inhomogeneous medium,
$
\Delta U + \frac{\omega^2}{c^2(z,x)}\,U 
= -\delta(z-z_s)\prod_{j=1}^{\bar{n}}\delta(x_j - x_{s,j}),
$
where $\omega$ denotes the angular frequency, $c(z,x)$ the wave speed, and $(z_s,x_s)$ the source location.  
Under the high-frequency (geometric optics) approximation, we employ the ansatz
$
U(z,x,\omega) = \big(A(x) + O(\tfrac{1}{\omega})\big)\exp\!\big(i\omega T(x)\big),
$
where $T(x)$ is the traveltime (phase) and $A(x)$ the amplitude~\cite{cor83,ockst02,qialeu04}.  
Substitution into the Helmholtz equation and collection of leading-order terms in $\omega$ yield the geometric optics relations
$
\|\nabla_x T\|^2 = \frac{1}{c^2}$ and $\nabla_x T \cdot \nabla_x A + \tfrac{1}{2} A\, \Delta_x T = 0$.
The first equation is the eikonal equation governing the phase, while the second describes amplitude transport along characteristics determined by the traveltime gradient $\nabla_x T$.

{
In adjoint-state traveltime tomography, the velocity field is reconstructed
from first-arrival traveltimes, and the gradient $\nabla T$ appears in
the adjoint transport operator used to compute sensitivities~\cite{leuqia0601}.
In our earlier work~\cite{leuqia0601}, we first compute the scalar
eikonal traveltime and then approximate its gradient by finite differences.
Several established methods are available for the scalar problem. These
include causal label-setting and fast marching methods
~\cite{tsi95,Sethian1996,popset97,set99siamrev}, as well as Gauss--Seidel
fast sweeping methods on Cartesian, triangular and surface grids
~\cite{tcoz03,kaooshqia04,zha05,qiazhazha07,qiazhazha0702,xuzhayon10,wonleu15}.
High-order, discontinuous Galerkin, factored and absolutely convergent
variants have also been developed
~\cite{zhashu03,zhazhaqia06,zclzs11,chechokao13,luoqiabur14,huzha24}.
Other approaches include fast semi-Lagrangian
~\cite{cristianifalcone2007fast}, fast iterative
~\cite{jeongwhitaker2008fast}, two-scale
~\cite{chaconvladimirsky2012fast}, and sparse-grid fixed-point sweeping
WENO methods~\cite{miksiszhang2024sparse}.
}

{
It has been observed that gradients recovered from first-order monotone
eikonal schemes may exhibit superconvergence in smooth regions, and that
compact upwind constructions can improve the gradient approximation~\cite{benluozha10}.
In this scalar-first approach, the accuracy of the recovered gradient
depends on both the scalar approximation and the subsequent differentiation.
Here, we instead derive and discretize an equation for the normalized
traveltime gradient $\mathbf n=\nabla\phi/\|\nabla\phi\|$, which
determines the direction used in ray tracing, amplitude transport and
adjoint inversion. By treating $\mathbf n$ as the primary variable, we
can analyze consistency directly for the orientation field and construct
the spatial reconstruction and time integration on the sphere.
}

{The use of a unit-vector unknown also relates the present formulation
to numerical methods for sphere-valued fields. Related geometric
discretizations include Lie-group and manifold
integrators~\cite{iserlesetal2000lie}, geodesic finite elements
~\cite{sander2012geodesic,sander2016higher}, constrained finite elements
for sphere-valued maps~\cite{bartels2009convergent}, structure-preserving
finite volumes~\cite{boscheri2021structure}, and intrinsic sphere-valued
Runge--Kutta methods~\cite{leuchalee24}. The present work follows this
geometric viewpoint in discretizing the eikonal orientation field, which
is used in ray tracing~\cite{julgub77,cer01,rawhausam07} and adjoint inversion
~\cite{leuqia0601,ple06,tncc09,lileuqia14,hqcwwl21}.
}

The remainder of this paper is organized as follows.  
Section~\ref{sec:background} reviews the background material and geometric preliminaries.  
Section~\ref{Sec:Proposal} presents the mathematical formulation and numerical construction of the proposed scheme.  
{The formulation is developed for smooth first-arrival regions.
Section~\ref{sec:Math_For} discusses compatibility with a scalar
traveltime, while Sections~\ref{sec:causality}--\ref{sec:initial}
describe the characteristic and initialization conditions used to
select the desired orientation field.
}{Section~\ref{sec:projection_analysis} then analyzes the comparison
methods to separate the contribution of the SLERP spatial discretization,
which is shared by normalized Euler and SFE, from that of the pseudo-time update.}
Section~\ref{sec:Num_Exp} reports numerical experiments and convergence studies.  
Finally, Section~\ref{sec:conc} summarizes the main findings and discusses possible extensions.

\section{Background}
\label{sec:background}

In this section, we summarize the mathematical and numerical ingredients underlying the proposed structure-preserving method for computing the traveltime gradient.  
We first review a related Godunov-type formulation for the linear advection equation, as used in fast sweeping and related iterative eikonal solvers~\cite{tcoz03,zha05,jeongwhitaker2008fast,leuqia0601}, to clarify the relation between upwind differencing, steady transport equations, and pseudo-time relaxation.  
We then introduce the spherical forward Euler (SFE) method~\cite{leuchalee24}, which forms the basis of spherical total-variation-diminishing Runge--Kutta (STVDRK) integrators for evolving ordinary differential equations on the unit sphere while preserving geometric structure and stability.

\subsection{Godunov upwinding and pseudo-time relaxation for the advection equation}
\label{SubSec:GodunovFSMAdvectionEqn}

We consider the one-dimensional scalar transport equation
$a(x) u_x = f(x)$, for a scalar function $u = u(x)$ with prescribed boundary conditions, where both the advection coefficient $a(x)$ and the source term $f(x)$ may vary spatially. 
Using a standard Godunov upwind discretization of the type underlying
sweeping and iterative eikonal methods~\cite{tcoz03,zha05,jeongwhitaker2008fast}, we obtain
\[
 a_i^+ \, \frac{{u_i-u_{i-1}}}{\Delta x}
 + a_i^- \, \frac{u_{i+1} - u_i}{\Delta x}
 = f_i ,
\]
where $a_i^{\pm} = \tfrac{1}{2}(a_i \pm |a_i|)$ represent the positive and negative characteristic contributions, respectively. 
This discretization naturally incorporates the direction of the underlying characteristics: when the flow is from left to right, i.e., $a_i > 0$, a backward difference is employed at $x = x_i$; otherwise, a forward difference is used.  

From this relation, the following iterative update can be derived:
\begin{equation}
\label{eq:Godunov_scalar_iterative_update}
    u_i^{k+1} =
    \frac{a_i^+}{|a_i|} u_{i-1}^{*}
    {-}\frac{a_i^-}{|a_i|} u_{i+1}^{*}
    + \frac{\Delta x}{|a_i|} f_i ,
\end{equation}
where the superscript $*$ denotes the value taken either from the current iteration $k$ or the next iteration $k+1$, depending on the current sweeping direction. 
This formulation highlights how information propagates along characteristics while maintaining computational efficiency through directional updates.

An alternative interpretation of this scheme arises by introducing an artificial time variable, transforming the steady transport equation into its time-dependent counterpart, $u_t + a(x) u_x = f(x)$. Applying a forward Euler discretization in time, together with the same upwind approximation in space, leads to
\begin{equation}
\label{eq:time_transport_iterative}
    u_i^{k+1} =
    \left(1 - \frac{|a_i| \Delta t}{\Delta x}\right) u_i^k
    + \frac{a_i^+ \Delta t}{\Delta x} u_{i-1}^k
    {-}\frac{a_i^- \Delta t}{\Delta x} u_{i+1}^k
    + \Delta t f_i.
\end{equation}
When the time step satisfies the optimal CFL condition $\Delta t = \Delta x / |a_i|$, this iteration nearly reduces to an update formula similar to \eqref{eq:Godunov_scalar_iterative_update}.
This connection motivates our later construction: rather than solving the spherical steady equation by a direct algebraic update, we use a Godunov-type spatial discretization together with a pseudo-time evolution toward steady state.

\subsection{Spherical Forward Euler (SFE)}
\label{SubSec:SFE}

We consider the ordinary differential equation (ODE) on the unit sphere $\mathbb{S}^2$, $\mathbf{p}'(t) = f(\mathbf{p}(t), t)$ with $\mathbf{p}(0) = \mathbf{p}_0 \in \mathbb{S}^2$, where the vector field $f: \mathbb{S}^2 \times [0,\infty) \to T(\mathbb{S}^2)$ is assumed to be Lipschitz continuous, ensuring existence and uniqueness of solutions.  The principal numerical difficulty lies in preserving the spherical constraint $\|\mathbf{p}(t)\|=1$, since standard time integrators do not, in general, guarantee that discrete approximations remain on $\mathbb{S}^2$. A common strategy applies classical Runge--Kutta (RK) schemes in $\mathbb{R}^3$ followed by projection onto the sphere at each time step. 
{Projection after a full RK step enforces the constraint at the step endpoint, while stage-projected schemes also constrain intermediate stages. In either case, accuracy and stability depend on the integrator as well as the projection.}

In~\cite{leuchalee24}, we introduced a class of explicit RK-type integrators tailored to ODEs on $\mathbb{S}^2$, achieving up to third-order accuracy.  
Rather than projecting, the method exploits the explicit expression of the exponential map on the sphere, thereby ensuring that the numerical solution remains on $\mathbb{S}^2$ at every stage. To illustrate the construction, consider advancing $\mathbf{p}^n \in \mathbb{S}^2$ with velocity 
$\mathbf{s} = f(\mathbf{p}^n, t^n)$ over a time step $h>0$.  
The update is defined through the exponential map,
$\mathbf{p}^{n+1} = \exp_{\mathbf{p}^n}(h\mathbf{s})$,
where $\exp_{\mathbf{p}}: T_{\mathbf{p}}\mathbb{S}^2 \to \mathbb{S}^2$ denotes the exponential map.  
On the unit sphere, this map admits the explicit formula
$
\exp_{\mathbf{p}}(\mathbf{s})
=
\cos(\|\mathbf{s}\|)\,\mathbf{p}
+
\sin(\|\mathbf{s}\|)\,
\frac{\mathbf{s}}{\|\mathbf{s}\|}$.
This leads to the \emph{spherical forward Euler (SFE)} scheme:
\[
\text{SFE: }
\left\{
\begin{array}{l}
\mathbf{s}^n = f(\mathbf{p}^n, t^n), \\[4pt]
\mathbf{p}^{n+1}
=
\cos(h\|\mathbf{s}^n\|)\,\mathbf{p}^n
+
\sin(h\|\mathbf{s}^n\|)\,
\dfrac{\mathbf{s}^n}{\|\mathbf{s}^n\|}.
\end{array}
\right.
\]

The SFE method is the manifold analogue of classical forward Euler, with motion constrained to $\mathbb{S}^2$ through the exponential map.  
By construction, $\|\mathbf{p}^{n+1}\|=1$ exactly, so no projection step is required.  
The scheme is locally second-order accurate and globally first-order accurate in time.  
Higher-order extensions are discussed in~\cite{leuchalee24}.
\section{Our Proposal}
\label{Sec:Proposal}

In this section, we present the mathematical formulation and numerical construction of the proposed method for computing the {normalized traveltime-gradient orientation field} on the unit sphere.  
We first derive the governing equation for the {orientation field} and examine its key structural properties.  
{Based on this formulation, we develop a nonlinear iterative scheme within the Godunov-type framework to obtain a numerical solution. 
The update rule is designed to preserve the unit-norm constraint of the
orientation vector at the discrete level, so that the numerical orientation remains on the unit sphere throughout the iteration.}

\subsection{Mathematical Formulation}
\label{sec:Math_For}

{We begin by rewriting the eikonal equation in the form
$\|\nabla\phi\|^2=\alpha^2$, where $\alpha(\mathbf{x}) := 1/c(\mathbf{x})$
denotes the slowness; see, e.g., \cite{Oshfed03,Sethian1996}
for standard Hamilton--Jacobi and eikonal formulations.}
Geometrically, this equation constrains the magnitude of the traveltime gradient $\nabla \phi$ to equal the local slowness $\alpha$, but it imposes no condition on the gradient's direction.  
To explicitly account for its orientation, we introduce the decomposition $\nabla \phi = \alpha \mathbf{n}$, where $\alpha = \|\nabla \phi\|$ represents the magnitude of the gradient, as required by the eikonal equation, and $\mathbf{n} \in \mathbb{S}^{n-1}$ is a unit vector specifying the orientation of $\nabla \phi$.  
This separation into magnitude and direction allows us to treat $\mathbf{n}$ as the primary variable for tracking the directional behavior of the traveltime gradient on the sphere.  
{For a unit numerical orientation $\mathbf n_h$, the reconstructed field $\mathbf q_h=\alpha\mathbf n_h$ satisfies $\|\mathbf q_h\|=\alpha$ up to round-off. For $\mathbf q_h$ to represent a scalar gradient, it must also satisfy the compatibility condition discussed below.}

To derive an equation for $\mathbf{n}$, we differentiate the eikonal equation and substitute the above decomposition. This yields
\begin{equation}
\alpha n_1 \mathbf{n}_x + \alpha n_2 \mathbf{n}_y + \alpha n_3 \mathbf{n}_z = \mathbf{r}(\mathbf{n}),
\label{eq:eikonal-vec-nonconservative}
\end{equation}
where $\mathbf{n} = (n_1, n_2, n_3)$ and $\mathbf{r}(\mathbf{n}) = \nabla \alpha - (\mathbf{n} \cdot \nabla \alpha)\, \mathbf{n}$. Since $\mathbf{n} \in \mathbb{S}^2$, the term $\mathbf{r}(\mathbf{n})$ represents the component of the slowness gradient $\nabla \alpha$ that is orthogonal to $\mathbf{n}$, i.e., the projection of $\nabla \alpha$ onto the tangent space $T_{\mathbf{n}}(\mathbb{S}^2)$.  

{
For a $C^2$ scalar viscosity solution satisfying $\nabla\phi\ne0$,
the normalized gradient $\mathbf n=\nabla\phi/\|\nabla\phi\|$
satisfies the orientation equation. Conversely, for $\mathbf q=\alpha\mathbf n$
to be locally a scalar gradient, we require $\nabla\times\mathbf q=0$.
The steady equation gives $(\mathbf q\cdot\nabla)\mathbf q=\nabla(\alpha^2/2)$.
Together with $\|\mathbf q\|=\alpha$, the vector identity
\[
 (\mathbf q\cdot\nabla)\mathbf q
 =\nabla(\|\mathbf q\|^2/2)-\mathbf q\times(\nabla\times\mathbf q)
\]
implies $\mathbf q\times(\nabla\times\mathbf q)=0$.
In two dimensions, the curl is perpendicular to the plane. Since
$\alpha>0$, the preceding relation therefore implies local integrability.
In three dimensions, however, the curl may be nonzero and parallel to $\mathbf q$.
For example, $\alpha=1$ and $\mathbf q=(\cos z,\sin z,0)$ satisfy the steady
equation, although $\nabla\times\mathbf q=-\mathbf q$.
This example shows that satisfying the steady orientation equation does
not by itself ensure scalar-gradient compatibility. It does not satisfy
the prescribed first-arrival boundary data used here. Since the discrete
scheme imposes no curl-free constraint, we examine compatibility under
mesh refinement in Table~\ref{tab:curl_refinement}.
}

For the remainder of this section, we analyze several properties of equation~\eqref{eq:eikonal-vec-nonconservative} and demonstrate its behavior through a simple example.
To this end, we consider the time-dependent evolution equation
\begin{equation}
\mathbf{n}_t
+ \alpha n_1 \, \mathbf{n}_x
+ \alpha n_2 \, \mathbf{n}_y
+ \alpha n_3 \, \mathbf{n}_z
= \mathbf{r}(\mathbf{n}),
\label{Eqn:TimeDependentN}
\end{equation}
whose steady-state solution satisfies~\eqref{eq:eikonal-vec-nonconservative}.  
Since this is a time-evolution equation, we must ensure that $\mathbf{n}(t, \mathbf{x}) \in \mathbb{S}^2$ for all $t > 0$ so that the solution remains physically meaningful.  

Taking the inner product of~\eqref{Eqn:TimeDependentN} with $\mathbf{n}$ gives
\[
\left( \tfrac{1}{2} \|\mathbf{n}\|^2 \right)_t
+ \tfrac{1}{2} \alpha n_1 \big(\|\mathbf{n}\|^2\big)_x
+ \tfrac{1}{2} \alpha n_2 \big(\|\mathbf{n}\|^2\big)_y
+ \tfrac{1}{2} \alpha n_3 \big(\|\mathbf{n}\|^2\big)_z
+ \big(\|\mathbf{n}\|^2 - 1\big)(\mathbf{n} \cdot \nabla \alpha) = 0.
\]
This implies that if the initial condition $\mathbf{n}(t=0, \mathbf{x}) = \mathbf{n}_0(\mathbf{x})$ satisfies $\|\mathbf{n}_0(\mathbf{x})\| = 1$, then $\|\mathbf{n}(t, \mathbf{x})\| = 1$ for all $t > 0$.
Consequently, the solution remains on the unit sphere throughout the evolution.
More importantly, this relation shows that each term in~\eqref{eq:eikonal-vec-nonconservative} lies naturally in the tangent space $T_{\mathbf{n}}(\mathbb{S}^2)$.
This geometric consistency provides a fundamental guiding principle for constructing structure-preserving discretizations of spherical-valued PDEs.

{For an example that is compatible away from the source, consider the
homogeneous eikonal equation with $\alpha(\mathbf x)=1$ and
$\phi(\mathbf0)=0$.}
The analytical viscosity solution to this equation is
$\phi(\mathbf{x}) = \|\mathbf{x}\|$, which immediately implies that the traveltime gradient field is
$\nabla \phi = \frac{\mathbf{x}}{\|\mathbf{x}\|} = \mathbf{n}(\mathbf{x})$.
It is straightforward to verify that this analytical solution satisfies
equation~\eqref{Eqn:TimeDependentN} for any $\mathbf{x}$ away from the origin.
Moreover, for any constant slowness $\alpha(\mathbf{x}) = \alpha_0$, we have $\mathbf{r}(\mathbf{n}) = \mathbf{0}$,
and the same solution $\mathbf{n}(\mathbf{x}) = \mathbf{x} / \|\mathbf{x}\|$ remains an exact steady-state
solution of~\eqref{Eqn:TimeDependentN}.

For the time-dependent evolution equation~\eqref{Eqn:TimeDependentN},
the homogeneous case $\alpha(\mathbf{x}) = 1$ reduces to the inviscid Burgers equation in three dimensions,
$\mathbf{n}_t + (\mathbf{n} \cdot \nabla) \mathbf{n} = \mathbf{0}$. This equation has been extensively studied in the literature.
Its connection to the
HJ framework and the qualitative behavior of its solutions are well established
(see, e.g.,~\cite{bul92,bor16,lar25}). 
In particular, the multidimensional generalized Burgers equation
\begin{equation}
\frac{\partial \mathbf{p}}{\partial t}
+ \left[ \nabla f(\mathbf{p}) \cdot \nabla \right] \mathbf{p}
- \nu \Delta \mathbf{p} = \mathbf{0},
\label{Eqn:GeneralizedBurgers}
\end{equation}
with $\nu > 0$, serves as an important model in the study of the Navier--Stokes equations~\cite{bor16}.
If we take the curl of~\eqref{Eqn:GeneralizedBurgers} and assume a potential field
$\mathbf{p}(t=0, \mathbf{x}) = \mathbf{p}_0(\mathbf{x}) = \nabla \phi^0$, the corresponding potential
$\phi(t, \mathbf{x})$ satisfies the viscous generalized HJ equation
$\frac{\partial \phi}{\partial t} + f(\nabla \phi) - \nu \Delta \phi = 0$. For the special nonlinear case $f(\mathbf{p}) = \|\mathbf{p}\|^2 / 2$,
this reduces to the viscous time-dependent eikonal equation, which has been widely investigated in the
level-set literature for modeling dynamic interfaces~\cite{oshset88,set96,Oshfed03}.
To align with our modeling framework, one could further decompose the gradient as $\grad\phi = \alpha \mathbf{n}$ and subsequently derive a governing equation for the vector field $\mathbf{n}$.

Although equation~\eqref{Eqn:GeneralizedBurgers} shares certain superficial similarities
with our formulation~\eqref{Eqn:TimeDependentN}, the two systems are fundamentally different in structure
and constraint. 
The generalized Burgers equation governs the evolution of the full traveltime gradient field
$\mathbf{p} = \nabla \phi$, without enforcing any restriction on its magnitude,
so that $\mathbf{p} \notin \mathbb{S}^2$ in general.
In contrast, our formulation focuses solely on the \emph{orientation} component,
requiring $\|\mathbf{n}\| = 1$ at all times, while the magnitude of the gradient is
automatically determined by the slowness $\alpha(\mathbf{x})$.
{
This separation between magnitude and direction is a basic geometric
feature of the proposed sphere-valued formulation.
}

\subsection{Numerical Discretization}
\label{sec:Num_For}

Godunov upwind discretizations are widely used for Hamilton--Jacobi equations and hyperbolic conservation laws because of their stability and robustness.  However, extending these ideas to spherical-valued PDEs presents several nontrivial challenges.
In particular, the numerical updates must remain consistent with the geometric constraints $\mathbf{n} \in \mathbb{S}^2$ and $\nabla \mathbf{n} \in T_{\mathbf{n}}(\mathbb{S}^2)$, where $T_{\mathbf{n}}(\mathbb{S}^2)$ denotes the tangent space of the sphere at $\mathbf{n}$.  

{
When Euclidean component-wise upwind differences are applied to a
sphere-valued field, the resulting approximation of the spatial
derivative need not lie in the tangent space, and the update need not
preserve unit length. To address these geometric constraints, we
construct the discretization directly from the geometry of the sphere.
}

{For a spherical-valued function $\mathbf n$ mapping each spatial
point into $\mathbb S^{n-1}$, we use SLERP endpoint derivatives to
approximate spatial derivatives in the tangent space and SFE to
advance the solution on the sphere. We first establish first-order
consistency, tangency and unit-length preservation.
Propositions~\ref{prop:geodesic_equilibrium} and~\ref{prop:sfe_nonlinear_ordering}
then describe preservation of curved equilibria and the behavior of
nonlinear relaxation in a one-cell model.}

\begin{figure}[!htb]
\centering
\includegraphics[width=\textwidth]{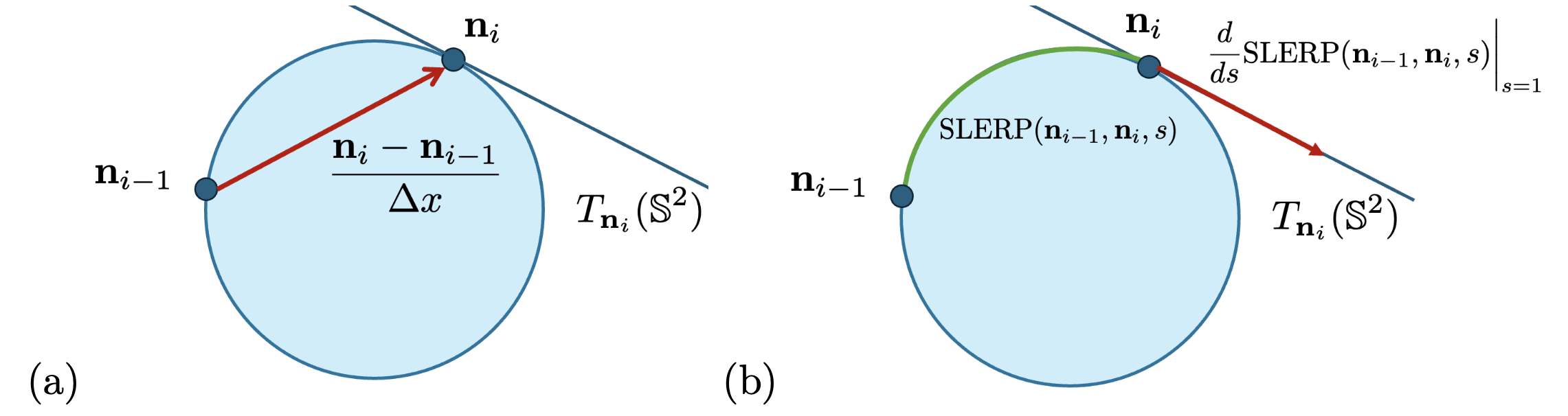}
\caption{
Numerical approximation of spatial derivatives. 
(a) Standard backward difference, illustrated by the red vector. 
(b) Proposed backward difference approximation based on spherical linear interpolation (SLERP). 
The SLERP interpolant is shown as the green curve, while the derivative at $\mathbf{n}_i$ is indicated by the red vector.
}
\label{fig:differencing}
\end{figure}

\subsubsection{The Challenges}

For simplicity, we first consider the one-dimensional case, where the solution $\mathbf{n}(x)$ is defined on the interval $x \in [0,1]$. In this setting, equation~\eqref{eq:eikonal-vec-nonconservative} reduces to $a(\mathbf{n}) \, \mathbf{n}_x = \mathbf{r}(\mathbf{n})$, where $a(\mathbf{n}) = \alpha n_1$ is the characteristic speed in the $x$-direction. Following the same idea introduced in Section~\ref{SubSec:GodunovFSMAdvectionEqn}, one might attempt to apply an upwind finite difference to approximate $\mathbf{n}_x$ and thereby derive a direct Godunov-type algebraic update for $\mathbf{n}_i$, as illustrated in Figure~\ref{fig:differencing}(a).
However, such a straightforward extension provides no guarantee that the unit-norm constraint is preserved.
The core difficulty arises because the conventional finite difference approximation does not ensure that the computed derivative lies in the tangent space $T_{\mathbf{n}}(\mathbb{S}^2)$.
{
As a result, component-wise updates do not, in general, preserve the
unit-sphere constraint.
}

To address this issue, we recall that the standard forward or backward difference formula can be interpreted as the derivative of a \emph{linear interpolant} between adjacent data points.  
Motivated by this observation, we propose to construct the analogous interpolant on $\mathbb{S}^2$ (specifically, the SLERP) and differentiate it to obtain a geometrically consistent finite difference approximation.  
The derivative of SLERP naturally represents the infinitesimal rate of change along a great-circle geodesic on $\mathbb{S}^2$, serving as the spherical counterpart of the linear finite difference in Euclidean space.

In particular, to interpolate two spherical data points $\mathbf{p}_a, \mathbf{p}_b \in \mathbb{S}^2$, we employ the SLERP formula~\cite{shoemake_85,sola}. 
Although a quaternion-based implementation~\cite{ham63} can be computationally more efficient, here we adopt its geometric form expressed via trigonometric functions:
\[
\mathrm{SLERP}(\mathbf{p}_a, \mathbf{p}_b, s)
= \frac{\sin[(1-s)\Omega]}{\sin \Omega}\, \mathbf{p}_a
+ \frac{\sin[s\Omega]}{\sin \Omega}\, \mathbf{p}_b, 
\qquad s \in [0,1],
\]
where $\cos \Omega = \mathbf{p}_a \cdot \mathbf{p}_b$ defines the angular distance between the two points along the great-circle arc. Building upon this formulation, we introduce a \emph{structure-preserving Godunov-type scheme} for equation~\eqref{eq:eikonal-vec-nonconservative}:
\begin{equation}
\label{eq:Godunov_vector_update}
\frac{a_i^+}{\Delta x} 
\frac{d}{ds} \mathrm{SLERP}(\mathbf{n}_{i-1}, \mathbf{n}_i, s)\bigg|_{s=1}
+ 
\frac{a_i^-}{\Delta x} 
\frac{d}{ds} \mathrm{SLERP}(\mathbf{n}_i, \mathbf{n}_{i+1}, s)\bigg|_{s=0}
= \mathbf{r}(\mathbf{n}_i).
\end{equation}
The two terms on the left-hand side represent geometrically consistent backward and forward difference operators for $\mathbb{S}^2$-valued functions, respectively, ensuring that all intermediate computations remain within the tangent space of the sphere.  

For instance, the derivatives of the SLERP interpolants can be written explicitly as
\begin{eqnarray*}
\frac{d}{ds} \mathrm{SLERP}(\mathbf{n}_{i-1}, \mathbf{n}_i, s)\bigg|_{s=1}
&=& \frac{\Omega_{i-1}}{\sin \Omega_{i-1}}
\left[ \cos \Omega_{i-1}\, \mathbf{n}_i - \mathbf{n}_{i-1} \right], \\[6pt]
\frac{d}{ds} \mathrm{SLERP}(\mathbf{n}_i, \mathbf{n}_{i+1}, s)\bigg|_{s=0}
&=& \frac{\Omega_{i+1}}{\sin \Omega_{i+1}}
\left[ \mathbf{n}_{i+1} - \cos \Omega_{i+1}\, \mathbf{n}_i \right],
\end{eqnarray*}
where $\cos \Omega_{i\pm1} = \mathbf{n}_i \cdot \mathbf{n}_{i\pm1}$.
Because SLERP interpolants always remain on $\mathbb{S}^2$, their derivatives are automatically confined to the tangent space $T_{\mathbf{n}_i}(\mathbb{S}^2)$, as illustrated in Figure~\ref{fig:differencing}(b).
{
Thus, the spatial derivative approximation in
\eqref{eq:Godunov_vector_update} retains the tangency that a conventional
component-wise finite difference generally loses.
}

{Equation~\eqref{eq:Godunov_vector_update} discretizes the orientation
equation directly. Its relation to projected componentwise differences
is analyzed in Section~\ref{sec:projection_analysis}.}

Based on the discretization~\eqref{eq:Godunov_vector_update}, one might obtain the following update formula:
\[
\mathbf{n}_i^{k+1} =
\left( 
\frac{a_i^+ \, \Omega_{i-1}^*}{\tan \Omega_{i-1}^*}
-
\frac{a_i^- \, \Omega_{i+1}^*}{\tan \Omega_{i+1}^*}
\right)^{-1}
\left[
\Delta x \, \mathbf{r}(\mathbf{n}_i^k)
+ 
\frac{a_i^+ \, \Omega_{i-1}^*}{\sin \Omega_{i-1}^*} \, \mathbf{n}_{i-1}^*
-
\frac{a_i^- \, \Omega_{i+1}^*}{\sin \Omega_{i+1}^*} \, \mathbf{n}_{i+1}^*
\right],
\]
where $\cos \Omega_{i\pm1}^* = \mathbf{n}_i^k \cdot \mathbf{n}_{i\pm1}^*$.

Although this iterative expression appears straightforward, it does not guarantee preservation of the unit-norm constraint $\mathbf{n}_i^{k+1} \in \mathbb{S}^2$.  
To illustrate this issue, consider the simple case $\alpha(x) = 1$ with the point-source boundary condition of the eikonal equation imposed at the leftmost point $x_0 = 0$.  
The exact solution of~\eqref{eq:eikonal-vec-nonconservative} is then $\mathbf{n}(x) = \mathbf{e}_1 = (1,0,0)$.  
Assume $a_i > 0$ so that $a_i^- = 0$, and let the latest updated value $\mathbf{n}_{i-1}^{k+1}$ satisfy $\|\mathbf{n}_{i-1}^{k+1}\| = 1$.  
Since $\mathbf{r}(x) = \nabla \alpha - (\mathbf{n}_i^k \cdot \nabla \alpha)\mathbf{n}_i^k = \mathbf{0}$ for any $\mathbf{n}_i^k$,  
the norm of the next iterate is given by
$\|\mathbf{n}_i^{k+1}\|^2 = {\sec^2 \Omega_{i-1}^{k+1}}$,
where $\cos \Omega_{i-1}^{k+1} = \mathbf{n}_i^k \cdot \mathbf{n}_{i-1}^{k+1} \neq 1$
unless $\mathbf{n}_i^k$ is already aligned with $\mathbf{n}_{i-1}^{k+1}$.
Consequently, ${\|\mathbf{n}_i^{k+1}\|>1}$ in general, indicating that the direct algebraic update above does not preserve the unit norm during intermediate updates. {Here $0<\Omega_{i-1}^{k+1}<\pi/2$; the formula is singular at $\pi/2$. This example concerns the algebraic alternative, not SFE.}
A straightforward remedy to maintain the unit-norm constraint would be to apply an explicit projection operator after each update step. 
{For a tangent Euler step, normalization also yields an update
along a great circle, but with a different angular step. We examine this
time-integration effect separately from the spatial approximation in
Section~\ref{sec:projection_analysis}.}

\subsubsection{Our Update Formula}

{To construct an intrinsic structure-preserving iteration}, we instead introduce an artificial time variable, following the pseudo-time viewpoint described in Section~\ref{SubSec:GodunovFSMAdvectionEqn}.  
The temporal derivative is discretized using the spherical forward Euler (SFE) method~\cite{leuchalee24}, which advances the solution along a geodesic curve on $\mathbb{S}^2$ via the exponential map.  
This leads to the following \emph{structure-preserving iterative scheme}:
\begin{equation}
\label{eq:Godunov_vector_update_final}
\begin{split}
\mathbf{s}_i &= \mathbf{r}(\mathbf{n}_i^k)
- \frac{a_i^+}{\Delta x}
\frac{d}{ds}\,\mathrm{SLERP}(\mathbf{n}_{i-1}^k, \mathbf{n}_i^k, s)\Big|_{s=1}
- \frac{a_i^-}{\Delta x}
\frac{d}{ds}\,\mathrm{SLERP}(\mathbf{n}_i^k, \mathbf{n}_{i+1}^k, s)\Big|_{s=0}, \\[4pt]
\mathbf{n}_i^{k+1} &= 
\cos(h\|\mathbf{s}_i\|)\,\mathbf{n}_i^k
+ 
\sin(h\|\mathbf{s}_i\|)\,\frac{\mathbf{s}_i}{\|\mathbf{s}_i\|},
\end{split}
\end{equation}
where $h = \Delta x / \max_i |a_i|$ and $a_i = \alpha_i n_1^k(x_i)$ in the one-dimensional setting.  
In the reported experiments, we take a fixed CFL number below one unless otherwise stated.  
This formulation describes a discrete evolution on the unit sphere that preserves the geometric structure of the solution.

It is important to note that both derivatives of the SLERP in~\eqref{eq:Godunov_vector_update_final} lie naturally on the tangent plane $T_{\mathbf{n}_i^k}(\mathbb{S}^2)$ at the current iterate $\mathbf{n}_i^k$.  
In addition, the term $\mathbf{r}(\mathbf{n}_i^k)$ is defined as the orthogonal projection of the local gradient $\nabla \alpha$ onto the tangent plane at $\mathbf{n}_i^k$, that is, the component of $\nabla \alpha$ orthogonal to $\mathbf{n}_i^k$, ensuring that $\mathbf{r}(\mathbf{n}_i^k) \in T_{\mathbf{n}_i^k}(\mathbb{S}^2)$.  
Consequently, every component contributing to $\mathbf{s}_i$ (namely, $\mathbf{r}(\mathbf{n}_i^k)$ and the two SLERP derivative terms) resides in the tangent space, so their linear combination $\mathbf{s}_i$ also lies in $T_{\mathbf{n}_i^k}(\mathbb{S}^2)$.  
This property is crucial because it guarantees that the update direction is always tangent to the unit sphere, satisfying the necessary condition for an evolution constrained to $\mathbb{S}^2$.  
Since the new iterate $\mathbf{n}_i^{k+1}$ is then obtained via the exponential map along this tangent direction, the unit-norm condition $\|\mathbf{n}_i^{k+1}\|=1$ is automatically maintained throughout the pseudo-time iteration.

\subsubsection{Consistency and Structure Preservation}\label{sec:consistency}
{The SLERP derivative is first-order consistent for smooth
sphere-valued fields, and the SFE update preserves the unit-sphere
constraint exactly. We establish these properties below.}

Let $\mathbf{p},\mathbf{q}\in\mathbb{S}^{2}$ with $\mathbf{p}\neq -\mathbf{q}$, and let $\Omega=\arccos(\mathbf{p}\cdot \mathbf{q})$.  
The SLERP curve $\boldsymbol{\gamma}(s)=\operatorname{SLERP}(\mathbf{p},\mathbf{q},s)$ is the unit-speed great-circle interpolation after reparametrization by the angular distance.  
Its endpoint derivative satisfies
\[
    \boldsymbol{\gamma}'(1)=\frac{\Omega}{\sin\Omega}\big(\cos\Omega\,\mathbf{q}-\mathbf{p}\big)\in T_{\mathbf{q}}\mathbb{S}^{2}, 
    \qquad
    \boldsymbol{\gamma}'(0)=\frac{\Omega}{\sin\Omega}\big(\mathbf{q}-\cos\Omega\,\mathbf{p}\big)\in T_{\mathbf{p}}\mathbb{S}^{2}.
\]
Moreover, if $\mathbf{q}=\mathbf{p}+\Delta x\,\mathbf{v}+O(\Delta x^2)$ for a smooth curve $\mathbf{p}(x)\in\mathbb{S}^{2}$ with $\mathbf{v}=\mathbf{p}_x\in T_{\mathbf{p}}\mathbb{S}^{2}$, then
\[
    \frac{1}{\Delta x}\frac{d}{ds}\operatorname{SLERP}(\mathbf{p},\mathbf{q},s)\Big|_{s=0}
    = \mathbf{v}+O(\Delta x).
\]
The same estimate holds for the backward derivative.  
Thus the proposed SLERP difference is a first-order consistent approximation of the spatial derivative of a smooth sphere-valued field and remains tangent to the sphere at the base point.

Since $\mathbf{r}(\mathbf{n})=\nabla \alpha-(\mathbf{n}\cdot\nabla\alpha)\mathbf{n}$ is the orthogonal projection of $\nabla\alpha$ onto $T_{\mathbf{n}}\mathbb{S}^2$, every term in $\mathbf{s}_i$ in \eqref{eq:Godunov_vector_update_final} lies in $T_{\mathbf{n}_i^k}\mathbb{S}^2$.  
The SFE update is exactly the Riemannian exponential map on the unit sphere, and hence
\[
    \|\mathbf{n}_i^{k+1}\|^2
    =\left\|\cos(h\|\mathbf{s}_i\|)\mathbf{n}_i^k
      +\sin(h\|\mathbf{s}_i\|)\frac{\mathbf{s}_i}{\|\mathbf{s}_i\|}\right\|^2=1,
\]
because $\mathbf{n}_i^k\cdot \mathbf{s}_i=0$.  
Therefore, up to round-off error, the iteration preserves the unit-sphere constraint without projection.  
This property is the discrete analogue of the continuous identity obtained by taking the inner product of \eqref{Eqn:TimeDependentN} with $\mathbf{n}$.

{These results show that the discretization is first-order consistent in space and
pseudo-time and preserves the unit-norm constraint exactly. Although a general convergence theorem is not established here, we examine the convergence behavior of the nonlinear relaxation numerically in Section~\ref{sec:Num_Exp}.}

\subsubsection{Causality and Relation to the Viscosity Solution}\label{sec:causality}

{The spatial operator uses the signs of the local characteristic
velocities to select upwind SLERP derivatives. In smooth regions, these
velocities follow the characteristic directions of the eikonal equation.
The resulting sphere-valued pseudo-time relaxation therefore retains
characteristic upwinding, although it does not inherit the scalar
ordering and comparison principle of a Godunov traveltime update.}

{To select the desired first-arrival branch, we prescribe compatible
data and compute the orientation in smooth regions. Source singularities,
shocks, caustics and characteristic crossings require separate treatment,
as the scalar traveltime may not have a single classical gradient there.
The following subsections describe the initialization and
source-neighborhood treatment used in our computations.}

\subsubsection{Initial Condition}\label{sec:initial}
The initialization procedure for the proposed scheme differs substantially from that of the scalar fast sweeping method. 
In the standard scalar Godunov fast sweeping method, all unknown grid points are typically initialized with a large numerical value, allowing monotone updates to propagate information outward from the boundary until convergence.
In contrast, the present spherical formulation evolves directional information rather than magnitudes.
Because the solution vector $\mathbf{n}$ represents the local characteristic or flow direction, its initial configuration must already be reasonably aligned with the expected physical field.

{
For the point-source tests, the corresponding homogeneous point-source
field provides a radially outward initial guess that agrees with the
expected first-arrival direction near the source. Neighboring vectors
also usually have similar directions, which is useful for the SLERP
construction: interpolation is not unique for an antipodal pair
$\mathbf p=-\mathbf q$, and its endpoint derivative becomes poorly
conditioned for nearly antipodal vectors.

Throughout the pseudo-time iteration, we keep directions prescribed by
the physical boundary condition or local point-source model fixed.
More generally, when no analytical interior direction is available, we
initialize from the homogeneous problem with the same boundary
condition. Where this choice is unavailable, the normalized numerical
gradient of a coarse scalar eikonal solution offers an alternative
initial guess away from singular regions.
}

\subsubsection{Point Source Condition} 

In the conventional eikonal formulation of the traveltime field, imposing the point-source boundary condition is straightforward: the traveltime is set to zero at the source location, and the scalar fast sweeping iterations naturally propagate this information across the domain in accordance with the wavefront's causality.  
In contrast, for the {orientation equation} considered in this work, the solution $\mathbf{n}$ is undefined at the point source because all characteristics (rays) emanate from that location, making the orientation field singular.  
To handle this singularity, it is necessary to design an appropriate boundary condition that accurately represents the local geometry of the wavefront near the source.

\begin{figure}[!htb]
\centering
\includegraphics[width=\textwidth]{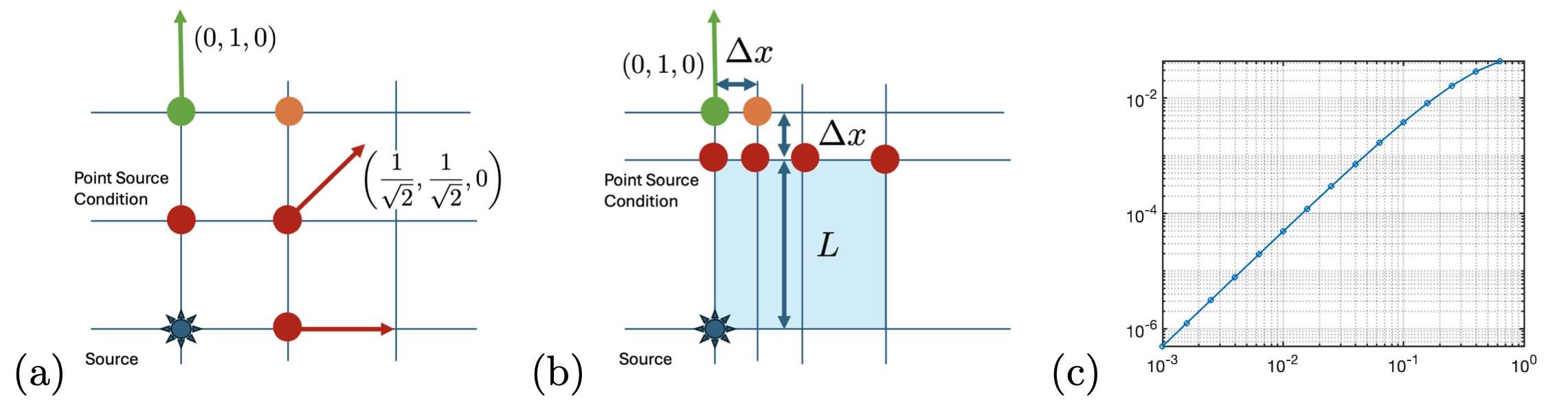}
\caption{Analysis of boundary condition treatments near the point source.  
(a) Initialization by assigning boundary conditions to the eight neighboring grid points surrounding the source.  
(b) Alternative approach imposing the point-source boundary condition within a small neighborhood $\Omega'$ around the source.  
(c) Numerical error $\|\mathbf{n}_{i,j} - \mathbf{n}_{\text{exact}}\|$ versus the non-dimensional grid spacing $h = \Delta x / L$, demonstrating first-order convergence.}
\label{fig:boundaryconditionanalysis}
\end{figure}

One practical approach to handling the source singularity is to exclude the point source itself from the computational domain and to prescribe the orientation field on its immediate neighbors.  
Specifically, we initialize the eight surrounding grid points according to $\mathbf{n}(\mathbf{x}_i) = \frac{\mathbf{x}_i}{\|\mathbf{x}_i\|}$,
which corresponds to assuming a locally homogeneous slowness distribution near the source (see Figure~\ref{fig:boundaryconditionanalysis}(a)).  
This initialization reflects the physical intuition that, in the immediate vicinity of the source, the wavefront expands isotropically, and thus the orientation vectors should radiate outward from the source.  
Because of the imposed local symmetry, the computed orientations along the coordinate axes are exact, for example, the green grid point yields $\mathbf{n} = (0,1,0)$, while along the diagonal direction (the red dot at the corner on the first ring), the exact solution is $\mathbf{n} = (\tfrac{1}{\sqrt{2}}, \tfrac{1}{\sqrt{2}}, 0)$.  
This setup provides an easily implemented approximation for initializing the iterative solver. 
To assess the numerical accuracy of this boundary prescription, we analyze the value at the orange grid point labeled $(i,j)$, which is determined using the two-dimensional extension of the SLERP-based steady discretization~\eqref{eq:Godunov_vector_update}.  
At this grid point, we seek a solution $(n_1, n_2, 0)$ that satisfies
\begin{equation}
n_1 \frac{\Omega_{i-1,j}}{\sin\Omega_{i-1,j}}
  \big( \cos \Omega_{i-1,j} \,\mathbf{n}_{i,j} - \mathbf{n}_{i-1,j} \big)
+ n_2 \frac{\Omega_{i,j-1}}{\sin\Omega_{i,j-1}}
  \big( \cos \Omega_{i,j-1} \,\mathbf{n}_{i,j} - \mathbf{n}_{i,j-1} \big)
= \mathbf{0},
\label{Eqn:PointSourceAnalysis}
\end{equation}
where $\cos \Omega_{i-1,j} = n_2$ and $\cos \Omega_{i,j-1} = \tfrac{1}{\sqrt{2}}(n_1 + n_2)$.  
Enforcing the unit-norm condition $n_2 = \sqrt{1 - n_1^2}$ leads to a root-finding problem of a single nonlinear function in $n_1$. 
Solving this numerically yields $n_1 = 0.484292917004127$ and $n_2 = 0.874905920965011$.  
The corresponding exact solution is $\mathbf{n}_{\text{exact}} = (\tfrac{1}{\sqrt{5}}, \tfrac{2}{\sqrt{5}}, 0)$, resulting in a deviation $\|\mathbf{n}_{i,j} - \mathbf{n}_{\text{exact}}\|_2 = 0.04190413$, which is unfortunately \textbf{independent of the grid size} $\Delta x$.
This analysis reveals an important limitation of the simple initialization strategy: the imposed boundary condition introduces a persistent, grid-independent discrepancy near the source.  
Consequently, the resulting numerical error does not diminish under mesh refinement, implying that this boundary treatment cannot achieve asymptotic convergence.  

Instead, we propose a more consistent boundary treatment by approximating the velocity model locally as homogeneous within a small neighborhood of the source.  
Specifically, we set $\alpha_0 = \alpha(\mathbf{x}_s)$ on a fixed subdomain 
$\Omega' = \{\mathbf{x} : \|\mathbf{x} - \mathbf{x}_s\|_{\infty} < L\}$ 
surrounding the point source.  
On the boundary $\partial \Omega'$ of this region, we impose the exact analytical boundary condition
$
\mathbf{n}_{\text{exact}} = 
\frac{\mathbf{x} - \mathbf{x}_s}{\|\mathbf{x} - \mathbf{x}_s\|_2},
$
which corresponds to a locally homogeneous eikonal solution with spherical wavefronts emanating from the source (see Figure~\ref{fig:boundaryconditionanalysis}(b)).  
This setup eliminates the singularity at the source and provides a physically consistent initialization for the numerical scheme.

To analyze the accuracy of this boundary construction, we consider the homogeneous case and examine convergence of the computed solution near the edge of $\Omega'$.  
For the grid point immediately one layer outside $\partial \Omega'$ above the source, the solution is exact due to the symmetry of both the geometry and the numerical stencil.  
Let $(x_i, y_j)$ denote the neighboring grid point on the right-hand side of this location (the orange dot in Figure~\ref{fig:boundaryconditionanalysis}(b)), and define the non-dimensionalized grid spacing $h = \Delta x / L$.  
The exact analytical solution at this grid point is given by $\mathbf{n}_{\text{exact}} = \frac{1}{\sqrt{h^2 + (1 + h)^2}} \,(h,\, 1+h,\, 0)$.
Numerically, the corresponding value computed by our scheme satisfies equation~\eqref{Eqn:PointSourceAnalysis} with 
$n_2 = \sqrt{1 - n_1^2}$, $\mathbf{n}_{i,j-1}=(1 + h^2)^{-1/2}(h,1,0)$, 
$\cos \Omega_{i-1,j} = (1 + h^2)^{-1/2}$, and 
$\cos \Omega_{i,j-1} = (1 + h^2)^{-1/2}(n_1 h + n_2)$.  
For each $h$, we solve the resulting nonlinear equation for $n_1$ and compare the numerical and exact solutions.  
Figure~\ref{fig:boundaryconditionanalysis}(c) shows the error $\|\mathbf{n}_{i,j} - \mathbf{n}_{\text{exact}}\|$ as a function of $h$, demonstrating a clear convergence as the grid is refined.  

This analysis confirms that the proposed localized homogeneous model yields a \emph{consistent and convergent} boundary condition for the point-source initialization.  
Unlike the direct eight-neighbor prescription discussed earlier, this approach eliminates the grid-independent error near the source and ensures that the numerical solution converges systematically with mesh refinement.

\subsubsection{Boundary Condition} 

Near the outer boundaries of the computational domain, equation~\eqref{eq:Godunov_vector_update_final} requires data outside the grid range.  
To handle this, the upwind differencing is modified to a one-sided form at the endpoints.  
Specifically, at the left and right boundaries ($i = 1$ and $i = m$), the velocity vectors are computed as
\begin{equation}
\label{eq:velocity_vector_at_boundary}
\begin{split}
\mathbf{s}_1 &= \mathbf{r}(\mathbf{n}_1^k)
+ \frac{|a_1|}{\Delta x}
\frac{d}{ds}\,\mathrm{SLERP}(\mathbf{n}_1^k, \mathbf{n}_2^k, s)\Big|_{s=0}, \\[4pt]
\mathbf{s}_m &= \mathbf{r}(\mathbf{n}_m^k)
- \frac{|a_m|}{\Delta x}
\frac{d}{ds}\,\mathrm{SLERP}(\mathbf{n}_{m-1}^k, \mathbf{n}_m^k, s)\Big|_{s=1},
\end{split}
\end{equation}
where the computational grid is defined as 
$x_{\min} = x_1 < x_2 < \cdots < x_m = x_{\max}$.  
This boundary treatment preserves the one-sided upwind information flow and maintains the spherical constraint at the domain edges.  

\subsubsection{Convergence Criterion}

In classical iterative solvers for scalar steady equations, convergence is assessed by monitoring the change between successive iterates of the scalar solution.  

{
For the controlled comparisons, we monitor the discrete steady
residual directly:
\[
 R^k=\max_{(i,j)\in\mathcal I_a}\|\mathbf s_{i,j}^k\|_2<\epsilon,
\]
where $\mathcal I_a$ contains the unfixed nodes and $\mathbf s^k$ is the
SLERP tangent residual evaluated from the complete field $\mathbf n^k$.
This stopping rule is used in Algorithm~\ref{alg:spfsm} and the controlled comparisons in
Sections~\ref{sec:projection_experiments}, \ref{sec:scalar_curved_cost}, \ref{sec:projection_3d_comparison} and \ref{sec:runtime_scaling}, with their
stated residual tolerances. The homogeneous, waveguide,
sinusoidal and three-dimensional model runs instead use the successive-iterate
stopping rule, with the corresponding histories labeled as update norms.

To relate the successive-iterate stopping rule to the residual criterion above, recall from the SFE update in~\eqref{eq:Godunov_vector_update_final} that the update moves $\mathbf{n}_{i,j}^k$ along a great circle by an angle $h\|\mathbf{s}_{i,j}^k\|_2$. Hence,
\begin{equation*}
\left\|
\mathbf{n}_{i,j}^{k+1}
-
\mathbf{n}_{i,j}^{k}
\right\|_2
=
2
\left|
\sin\left(
\frac{
h\|\mathbf{s}_{i,j}^{k}\|_2
}{2}
\right)
\right|.
\end{equation*}
For sufficiently small pseudo-time steps,
\begin{equation*}
\left\|
\mathbf{n}_{i,j}^{k+1}
-
\mathbf{n}_{i,j}^{k}
\right\|_2
=
h\|\mathbf{s}_{i,j}^{k}\|_2
+
O\left(
h^3\|\mathbf{s}_{i,j}^{k}\|_2^3
\right).
\end{equation*}
Thus, the successive-iterate norm provides a useful indication of convergence toward a discrete steady state. However, since the pseudo-time step $h$ depends on the mesh spacing, a fixed tolerance on the update norm does not correspond to the same residual tolerance on different meshes.
}

\subsubsection{Algorithm and Computational Considerations}\label{sec:cost}
For completeness, we summarize the complete two-dimensional implementation in Algorithm~\ref{alg:spfsm}.  
The three-dimensional version is obtained by adding the corresponding $z$-direction SLERP derivatives.

\begin{algorithm}[htbp]
\caption{Structure-preserving pseudo-time method for the orientation field}
\label{alg:spfsm}
\begin{algorithmic}[1]
\STATE Prescribe boundary or source-neighborhood values of $\mathbf{n}$.
\STATE Initialize all remaining grid values by a physically consistent orientation field.
\FOR{$k=0,1,2,\ldots$ until convergence}
        \FOR{each unfixed grid point $(i,j)$}
            \STATE Evaluate the upwind coefficients $a_{i,j}^{\pm}$ and $b_{i,j}^{\pm}$ from $\alpha_{i,j}\mathbf{n}_{i,j}^k$.
            \STATE Compute the SLERP endpoint derivatives in the selected upwind directions.
            \STATE Form the tangent vector $\mathbf{s}_{i,j}$ by the multidimensional analogue of \eqref{eq:Godunov_vector_update_final}.
            \STATE Choose the pseudo-time step $h$ from the multidimensional CFL-type bound.
            \STATE Update $\mathbf{n}_{i,j}^{k+1}=\exp_{\mathbf{n}_{i,j}^{k}}(h\mathbf{s}_{i,j})$ using the SFE formula.
        \ENDFOR
    \STATE {Recompute $\mathbf s^{k+1}$ from the complete new field and stop if
    $\max_{(i,j)\in\mathcal I_a}\|\mathbf s_{i,j}^{k+1}\|_2<\epsilon$.}
\ENDFOR
\end{algorithmic}
\end{algorithm}

{
For a grid with $N$ points, each pseudo-time iteration of the proposed method requires $O(N)$ local operations.
}  
{Each local update evaluates dot products and trigonometric factors
for the SLERP derivative and exponential map.}
In practice, the formulas are evaluated with their Taylor expansions when the angular distance is small, for example $\Omega/\sin\Omega=1+O(\Omega^2)$, to avoid loss of accuracy and unnecessary cost.  When $\Omega$ is below a prescribed tolerance, we replace the SLERP derivative by its tangent linearization.  
The antipodal case $\mathbf{p}=-\mathbf{q}$ is not unique for SLERP.  The method therefore assumes that neighboring directions are not antipodal; in the implementation, nearly antipodal neighboring pairs are treated as an indication that the current grid resolution or initialization is inadequate near a directional discontinuity.

{
The total work for $K_N$ iterations is $O(NK_N)$, so the overall cost
also depends on how the iteration count grows under mesh refinement.
With $N=O(\Delta x^{-d})$, an iteration bound $K_N=O(\Delta x^{-1})$
would give $O(N^{1+1/d})$ work. In Section~\ref{sec:runtime_scaling},
we examine this conditional scaling estimate using measured iteration
counts and total runtimes. Costs on fixed meshes and at prescribed
accuracy are compared in Sections~\ref{sec:projection_experiments},
\ref{sec:scalar_curved_cost} and~\ref{sec:projection_3d_comparison}.

}

{
\subsection{Analysis of the Comparison Methods}
\label{sec:projection_analysis}

We introduce comparison methods to examine the spatial reconstruction
and the sphere-preserving time update separately. We write
$P_{\mathbf n}=I-\mathbf n\mathbf n^T$ and let $D_j^{\rm up}$ denote the
componentwise one-sided difference selected by the sign of $\alpha n_j$.
We denote the corresponding SLERP endpoint derivative divided by the mesh
spacing by $L_j^{\rm up}$. The tangent spatial residuals are then given by
\begin{equation}\label{eq:projection_residuals}
\begin{split}
 F_{\rm proj}(\mathbf n)&=\mathbf r(\mathbf n)
       -\alpha\sum_{j=1}^d n_j P_{\mathbf n}D_j^{\rm up}\mathbf n,\\
 F_{\rm S}(\mathbf n)&=\mathbf r(\mathbf n)
       -\alpha\sum_{j=1}^d n_j L_j^{\rm up}\mathbf n.
\end{split}
\end{equation}
We name each method by its spatial residual (FD, PFD or SLERP) and
time update, normalized Euler (NE) or spherical forward Euler (SFE).
FD uses ambient componentwise differences, whereas PFD projects their residual
onto the tangent space before applying the NE normalization.
The paired comparisons in Table~\ref{tab:projection_variants} isolate
the effects of residual projection (FD--NE/PFD--NE), SLERP derivatives
(PFD--NE/SLERP--NE), and the time update (SLERP--NE/SLERP--SFE).
These comparisons concern direct orientation solvers; we introduce
the scalar-recovery workflows separately below. We denote the pseudo-time
step by $\Delta\tau$ (called $h$ in~\eqref{eq:Godunov_vector_update_final})
and the uniform spatial spacing by $\Delta x$.

\begin{table}[!htb]
\centering\small

\begin{tabular}{@{}l p{0.47\textwidth} p{0.33\textwidth}@{}}
\toprule
Method & Spatial residual & Pseudo-time update\\
\midrule
FD--NE & Componentwise upwind advection with $\mathbf r(\mathbf n)$;
    no projection of the full residual
  & Ambient Euler followed by normalization\\
PFD--NE & Componentwise upwind, projected tangent: $F_{\rm proj}$
  & Euler followed by normalization\\
SLERP--NE & Present SLERP derivatives: $F_{\rm S}$
  & Euler followed by normalization\\
SLERP--SFE & Present SLERP derivatives: $F_{\rm S}$
  & Present spherical forward Euler\\
\bottomrule
\end{tabular}
\caption{Controlled comparison of the spatial residual and pseudo-time update.
FD--NE, PFD--NE, and SLERP--NE are comparison baselines; SLERP--SFE is the proposed method.
All four methods preserve unit length up to round-off error.}
\label{tab:projection_variants}
\end{table}

\subsubsection{Spatial discretization.}
Method FD--NE uses $G=\mathbf r-\alpha\sum_j n_jD_j^{\rm up}\mathbf n$, with
$P_{\mathbf n}G=F_{\rm proj}$. If $1+\Delta\tau\,\mathbf n\cdot G>0$, its
normalized update is equivalent to normalized Euler applied to $F_{\rm proj}$ with
effective step $\Delta\tau/(1+\Delta\tau\,\mathbf n\cdot G)$.
Thus, in this regime, FD--NE and PFD--NE have the same steady
equation with zero tangential residual, although their iteration histories can differ.

To compare PFD--NE with SLERP--NE, consider unit vectors $\mathbf n,\mathbf q$
with angular separation $\Omega<\pi$,
\begin{equation}\label{eq:chord_log_comparison}
 \log_{\mathbf n}(\mathbf q)
 =\frac{\Omega}{\sin\Omega}\,P_{\mathbf n}(\mathbf q-\mathbf n).
\end{equation}
At $\Omega=0$, we use the limiting expression. Geometrically, projected
componentwise differences measure the tangent component of a chord,
whereas SLERP derivatives also include its angular correction. For smooth fields with
$\Omega=O(\Delta x)$, the two derivative approximations differ by
$O((\Delta x)^2)$. Both methods are generally first order; however,
the angular correction allows exact preservation of the nonconstant
equilibria described in the following proposition.

\begin{proposition}[Preservation of a geodesic equilibrium]
\label{prop:geodesic_equilibrium}
Let $\mathbf a,\mathbf b\in\R^d$ be orthonormal, $k\ne0$, and
$u=\mathbf a\cdot\mathbf x$, $v=\mathbf b\cdot\mathbf x$,
$\theta=\theta_0+ku$. Define
\begin{equation}\label{eq:geodesic_family}
 \alpha=e^{kv},\qquad
 \phi=\frac{e^{kv}}{k}\sin\theta,\qquad
 \mathbf n=\cos\theta\,\mathbf a+\sin\theta\,\mathbf b.
\end{equation}
Then $\nabla\phi=\alpha\mathbf n$. On a Cartesian mesh with spacings
$\Delta x_j$, the SLERP spatial residual vanishes on this sampled field
wherever the required one-sided neighbors are available, provided
$|ka_j\Delta x_j|<\pi$ and $\nabla\alpha$ is evaluated analytically.
Method SLERP--NE and the proposed method both preserve this equilibrium in exact
arithmetic when the prescribed boundary data are consistent with it.
\end{proposition}

\begin{proof}
Set $\mathbf t=-\sin\theta\,\mathbf a+\cos\theta\,\mathbf b$.
Direct differentiation gives $\partial_j\mathbf n=ka_j\mathbf t$ and
$\mathbf r(\mathbf n)=k\alpha\cos\theta\,\mathbf t$.
The signed angular increment on a grid link is $ka_j\Delta x_j$.
Its SLERP endpoint derivative, divided by $\Delta x_j$, is exactly
$ka_j\mathbf t$, for either one-sided stencil under the stated angle
condition. Hence
$\alpha\sum_j n_jL_j^{\rm up}\mathbf n
=k\alpha(\mathbf n\cdot\mathbf a)\mathbf t
=\mathbf r(\mathbf n)$, proving the assertion.
\end{proof}

For the coordinate-aligned two-dimensional case,
$\mathbf a=(1,0)$, $\mathbf b=(0,1)$, the projected chord derivative is
$[\sin(k\Delta x)/\Delta x]\mathbf t$ instead of $k\mathbf t$. Thus the
sampled exact field has the nonzero equilibrium defect
\begin{equation}\label{eq:geodesic_chord_defect}
 F_{\rm proj}(\mathbf n)
 =\alpha\cos\theta\left(k-\frac{\sin(k\Delta x)}{\Delta x}\right)\mathbf t
 =O((\Delta x)^2).
\end{equation}

{
The defect in~\eqref{eq:geodesic_chord_defect} explains the background
bias retained by the projected-chord discretization at convergence.
To assess its significance for nearby fields,
Section~\ref{sec:perturbation_resolution} compares this bias with the
physical directional change induced by perturbing the equilibrium.
}

\subsubsection{Pseudo-time update.}
For a unit vector $\mathbf n$ and a nonzero tangent vector $\mathbf s$, put
$z=\Delta\tau\|\mathbf s\|$ and $\mathbf t=\mathbf s/\|\mathbf s\|$.
Normalized Euler (NE) and SFE advance along the same great circle:
\begin{equation}\label{eq:normalized_euler_angle}
\begin{split}
 \mathbf n_{\rm NE}^{+}
 &=\frac{\mathbf n+z\mathbf t}{\sqrt{1+z^2}}
   =\cos(\arctan z)\mathbf n+\sin(\arctan z)\mathbf t,\\
 \mathbf n_{\rm SFE}^{+}
 &=\cos z\,\mathbf n+\sin z\,\mathbf t,\\
 \|\mathbf n_{\rm SFE}^{+}-\mathbf n_{\rm NE}^{+}\|
 &=2\left|\sin\!\left(\frac{z-\arctan z}{2}\right)\right|
   =\frac{z^3}{3}+O(z^5).
\end{split}
\end{equation}
Both updates preserve the sphere and leave every zero of their common
residual unchanged. For fixed $\Delta\tau$, their tangent-space linearization at a
zero-residual equilibrium $\mathbf n_*$ is
$I+\Delta\tau\,D F_{\rm S}(\mathbf n_*)$, whenever $F_{\rm S}$ is
differentiable or the upwind branch is fixed. Consequently, both updates have
the same local linear stability tests and asymptotic linear factors in two and
three dimensions. Their nonlinear transients can nevertheless differ,
as we illustrate in the following reduction.

\emph{A conditional nonlinear advantage of SFE.}
To compare the nonlinear angular updates explicitly, we consider a
reduction of the same spatial residual.
On $[0,\Delta x]^2$, we take $\alpha(y)=e^{ky}$, $k\ge0$, with zero
incoming traveltime at $y=0$. The exact first-arrival solution is
$\phi(y)=(e^{ky}-1)/k$, with limit $\phi=y$ at $k=0$, and
$\mathbf n_*=(0,1)$. Vertical paths attain the optical-length lower
bound given by the change in $\phi$. We use a $2\times2$ grid, fix the
two bottom orientations to $\mathbf n_*$, and initialize both top
nodes at $(\sin e_0,\cos e_0)$, where $0<e_0<\pi/2$.
The two top orientations remain equal. Writing
$\mathbf n=(\sin e,\cos e)$, $\mathbf t=(\cos e,-\sin e)$,
$\beta=k\Delta x$, and $\mu=\Delta\tau e^{k\Delta x}/\Delta x$, the
unchanged SLERP residual then reduces to
\begin{equation}\label{eq:sfe_angular_residual}
 F_{\rm S}(\mathbf n)
 =-\frac{e^{k\Delta x}}{\Delta x}g_\beta(e)\mathbf t,
 \qquad g_\beta(e)=\beta\sin e+e\cos e.
\end{equation}
Indeed, the horizontal derivative vanishes, the logarithm to the bottom
neighbor is $-e\mathbf t$, and
$\mathbf r=-ke^{k\Delta x}\sin e\,\mathbf t$.
For the proposed SFE update and the normalized Euler update of method SLERP--NE,
the angular iteration maps are
\begin{equation}\label{eq:sfe_angular_maps}
 T_{\rm SFE}(e)=e-\mu g_\beta(e),\qquad
 T_{\rm NE}(e)=e-\arctan\!\bigl(\mu g_\beta(e)\bigr).
\end{equation}

\begin{proposition}[Angular-error ordering for a one-cell relaxation]
\label{prop:sfe_nonlinear_ordering}
For the data above, if $0<\mu(\beta+1)\le1$, the iterates of
\eqref{eq:sfe_angular_maps} from the same $e_0$ satisfy
$0\le e_{\rm SFE}^m<e_{\rm NE}^m\le e_0$ for every $m\ge1$
in exact arithmetic. Both sequences decrease to zero. Thus SFE takes
no more iterations to reach any prescribed angular-error threshold.
\end{proposition}

\begin{proof}
For $0\le e\le\pi/2$,
$0\le g_\beta(e)\le(\beta+1)e$, so
$0\le T_{\rm SFE}(e)\le T_{\rm NE}(e)\le e$.
For $0<e<\pi/2$, $g_\beta(e)>0$ and
$T_{\rm SFE}(e)<T_{\rm NE}(e)<e$ because $\arctan z<z$ for $z>0$.
Moreover, $T_{\rm SFE}'(e)=1-\mu[(\beta+1)\cos e-e\sin e]\ge0$
under the stated step condition. Monotonicity and induction give the
strict ordering.
Both decreasing sequences have limits, and the only zero of
$g_\beta$ in $[0,e_0]$ is zero.
\end{proof}

Both maps have derivative $1-\mu(\beta+1)$ at zero. At
$\mu=1/(\beta+1)$ this common linear factor vanishes, but
\begin{equation}\label{eq:sfe_cubic_constants}
 T_{\rm SFE}(e)=\frac{\beta+3}{6(\beta+1)}e^3+O(e^5),\qquad
 T_{\rm NE}(e)=\frac{3\beta+5}{6(\beta+1)}e^3+O(e^5).
\end{equation}
For $\beta=3$, the leading coefficients are $1/4$ for SFE and $7/12$
for NE. Thus, the iteration converges cubically at the critical step,
although both time integrators remain first order. Since $g_\beta$ need
not be monotone, the ordering established in the proposition applies to
angular errors rather than residuals. In Section~\ref{sec:sfe_nonlinear_experiment},
we examine this nonlinear angular-error behavior numerically.

\par}

\section{Numerical Experiment}
\label{sec:Num_Exp}

{The experiments examine equilibrium preservation, directional
accuracy and computational cost in two and three dimensions. We begin
with a curved orientation equilibrium and signed perturbations about
it, then isolate the pseudo-time update. Comparisons with scalar-first
recovery examine orientation accuracy.
Waveguide, sinusoidal and three-dimensional models extend the tests to
other media, followed by a computational-scaling study.}
We omit the one-dimensional test case, as the vector field $\mathbf{n}$ becomes trivial under the unit-length constraint, taking only the values $\pm 1$.

{
In the following experiments, we denote the proposed method by SLERP--SFE
and compare it with the orientation controls defined in
Table~\ref{tab:projection_variants}. 
We introduce the scalar-first comparison methods in
Section~\ref{sec:scalar_recovery_comparison}.
}

\FloatBarrier
{
\subsection{Projection Ablation: Spatial Preservation and Nonlinear Convergence}
\label{sec:projection_experiments}

We compare FD--NE, PFD--NE, SLERP--NE, and SLERP--SFE as defined in
Table~\ref{tab:projection_variants}. Unless stated otherwise, the
computations use the same medium and initial and boundary data for
all methods, with the pseudo-time step given by
\begin{equation}\label{eq:ablation_timestep}
 \Delta\tau=\frac{0.9}{(\max_i\alpha_i)\sum_{j=1}^d1/\Delta x_j}.
\end{equation}

\subsubsection{Spatial Discretization in Two Dimensions}

To examine the preservation property established in
Proposition~\ref{prop:geodesic_equilibrium}, we consider the domain
$[0,1]^2$ in scaled coordinates. We take $k=1.2$ and $\theta_0=0.15$
and use the coordinate-aligned field~\eqref{eq:geodesic_family}, given by
\begin{equation}\label{eq:graded_2d}
 \alpha(x,y)=e^{ky},\quad c(x,y)=e^{-ky},\quad
 \phi(x,y)=\frac{e^{ky}}{k}\sin(\theta_0+kx),\quad
 \mathbf n=(\cos(\theta_0+kx),\sin(\theta_0+kx)).
\end{equation}
In this example, angles are measured from the positive $x$-axis.
We prescribe the exact incoming orientations on $x=0$ and $y=0$
and apply one-sided outgoing stencils on the remaining edges.
At every unfixed node, the initial orientation is the nonexact
constant direction $(\cos\theta_0,\sin\theta_0)$. Since both exact
components are positive, the backward characteristics reach the
incoming boundary. The relation $\nabla\phi=\alpha\mathbf n$ implies
that the optical length bounds the change in $\phi$, with equality
along these characteristics. Thus, the corresponding incoming
traveltimes select the first-arrival solution.

For the graded-medium tests, the iteration is terminated when the
maximum tangential residual reaches $10^{-10}$. We apply the same
criterion to all methods, changing only the monitoring rule.
This allows us to compare the methods independently of the update
norms associated with different steps. To evaluate the accuracy,
we use the RMS angular error
\begin{equation}\label{eq:ablation_rms_angle}
 E_{\rm ang}=\left(\frac{1}{N_a}\sum_{i\in\mathcal I_a}
   \operatorname{atan2}\!\left(
    \|\mathbf n_i^*-(\widehat{\mathbf n}_i\cdot\mathbf n_i^*)
                  \widehat{\mathbf n}_i\|,
    \widehat{\mathbf n}_i\cdot\mathbf n_i^*\right)^2\right)^{1/2},
\end{equation}
where $\mathcal I_a$ contains the $N_a$ unfixed nodes including outgoing
boundary nodes, $\mathbf n_i^*$ is the exact unit orientation, and
$\widehat{\mathbf n}_i=\mathbf n_i/\|\mathbf n_i\|$ is used only to
evaluate this diagnostic. All angles are in radians.

We next examine how the methods behave away from the geodesic family
by applying signed perturbations to the preserved equilibrium.
Specifically, we use
\begin{equation}\label{eq:graded_perturbation}
 f_\varepsilon(x)=\frac{\sin(\theta_0+kx)}{k}
            +\frac{\varepsilon\sin(2\pi x)}{2\pi},\qquad
 \phi_\varepsilon=e^{ky}f_\varepsilon(x),\qquad
 \alpha_\varepsilon=\|\nabla\phi_\varepsilon\|,
\end{equation}
with exact orientation $\nabla\phi_\varepsilon/\alpha_\varepsilon$ and
analytic $\nabla\alpha_\varepsilon$; no balancing term is added.
We test $\varepsilon=0,\allowbreak\pm0.001,\allowbreak\pm0.005,\allowbreak\pm0.05$ on
$\Delta x=0.1,0.05,0.025,0.0125$, adding $\Delta x=0.00625$ for
$\varepsilon=-0.001,0,0.001$. 
Figure~\ref{fig:graded_orientation_fields}(a,b) compares the fields for
$\varepsilon=0,0.05$; panel (c) shows the signed angle changes for
$\varepsilon=\pm0.05$ along $y=0.5$. Reference rays satisfy $d\mathbf X/ds=\mathbf n^*(\mathbf X)$
with arc length $s$ and the analytic orientation $\mathbf n^*$.
Here $\theta_\varepsilon^{\rm prop}$ denotes the computed angle
from the positive $x$-axis. These changes describe the response to the
medium; errors relative to each exact field appear in
Figure~\ref{fig:projection_equilibrium}.

\begin{figure}[!htb]
\centering

\includegraphics[width=\textwidth]{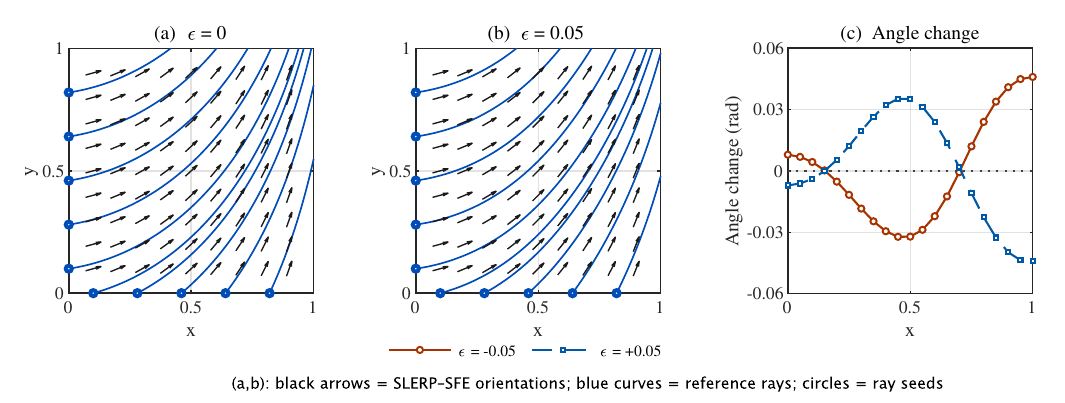}
\caption{Computed orientations and medium-induced angle changes for the
two-dimensional family~\eqref{eq:graded_perturbation}, with $k=1.2$,
$\theta_0=0.15$, $\Delta x=0.05$, analytic $\nabla\alpha_\varepsilon$,
and residual tolerance $10^{-10}$.
(a,b) Computed orientation fields for $\varepsilon=0,0.05$, respectively. Black arrows show
the computed orientations of the proposed method at every second interior grid node,
while blue curves are reference characteristics
integrated from the exact orientation field, starting at the circles
on $x=0$ and $y=0$ and ending at the outgoing boundary. 
(c) Signed angle changes for $\varepsilon=-0.05,+0.05$ relative to the
computed $\varepsilon=0$ field along $y=0.5$. The dotted line represents the unperturbed case $\varepsilon=0$. This compares the responses in different media rather than errors relative to the exact solution.}
\label{fig:graded_orientation_fields}
\end{figure}
\FloatBarrier

Figure~\ref{fig:projection_equilibrium}(a) confirms equilibrium
preservation: SLERP--NE and SLERP--SFE reach the exact field to the
stopping accuracy, while FD--NE and PFD--NE reach a nearby discrete
equilibrium. At $\Delta x=0.1$, FD--NE and PFD--NE each require
205 iterations, taking 3.571 and 3.750 ms, respectively; SLERP--NE
and SLERP--SFE each require 204 iterations, taking 5.562 and 6.441 ms.
The SLERP residual evaluated on the exact field remains below
$10^{-13}$ for $\Delta x=0.00625,0.0125,0.025,0.05,0.1$.
We also start from the exact field and take 100 steps. SLERP--NE
and SLERP--SFE preserve it to round-off, whereas FD--NE and PFD--NE
move away from it.

\begin{figure}[!htb]
\centering

\includegraphics[width=\textwidth]{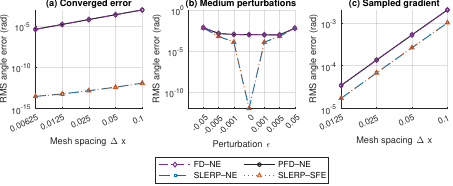}
\caption{Spatial comparison for the graded medium, with
$k=1.2$, $\theta_0=0.15$, and residual tolerance $10^{-10}$.
(a) RMS angular error after solving from nonexact constant data.
(b) Signed perturbations~\eqref{eq:graded_perturbation} at $\Delta x=0.1$;
the horizontal positions enumerate the tested values.
(c) Errors with the same centered-difference approximation of
$\nabla\alpha$ supplied to every method. The curves for FD--NE and PFD--NE nearly
coincide, as do those for SLERP--NE and the proposed method. Tolerance-level
errors in (a) do not define convergence rates.}
\label{fig:projection_equilibrium}
\end{figure}

{
\paragraph{Resolution of small directional changes.}
\label{sec:perturbation_resolution}
Let $\theta_\varepsilon^*$ and $\theta_\varepsilon^h$ be the exact
and computed angles, $\theta_{\rm eq}=\theta_0+kx$, and $\theta_0^h$
the computed unperturbed angle. On the unfixed nodes, the physical
change has RMS size $S_\varepsilon=
\|\theta_\varepsilon^*-\theta_{\rm eq}\|_{\rm RMS}$ for
$\varepsilon\ne0$. Define $R_{\rm abs}=E_{\rm ang}/S_\varepsilon$ and
the paired-response ratio
\begin{equation}\label{eq:perturbation_resolution}
 R_{\rm pair}=
 \frac{\|(\theta_\varepsilon^h-\theta_0^h)
       -(\theta_\varepsilon^*-\theta_{\rm eq})\|_{\rm RMS}}
      {S_\varepsilon}.
\end{equation}
For $\Delta x=0.1$, the perturbations $\varepsilon=\pm0.001$ produce
a physical change of $S_\varepsilon\simeq5.42\times10^{-4}$ rad.
For the negative and positive signs, FD--NE/PFD--NE give
$R_{\rm abs}=1.983,1.799$, while both SLERP methods give
$0.234,0.233$. The paired ratios for the componentwise and SLERP
methods are $R_{\rm pair}\simeq0.237$ and $0.234$, respectively.
Thus, subtracting each computed background cancels much of the
componentwise bias. These results show that SLERP improves the
absolute orientation accuracy near the equilibrium by suppressing
this background error.
}

At the $10^{-4}$-rad target, the SLERP methods require less
computational time than FD--NE and PFD--NE for both signs of the
small perturbation (Table~\ref{tab:projection_cost}). The comparison
changes for larger perturbations: $\varepsilon=-0.05$ favors SLERP,
whereas at $+0.05$ PFD--NE is slightly more accurate. For
$\varepsilon=+0.001$, imposing the stricter $10^{-5}$ target
requires every method to use the tested $\Delta x=0.00625$ grid.
In this case, FD--NE and PFD--NE have lower computational costs.

We next replace the analytic $\nabla\alpha$ in the unperturbed
example by the same centered difference of sampled slowness for
all methods. At the edges, we supply exact slowness samples one
mesh interval outside the domain. The resulting derivative is
$\partial_y^{\Delta x}\alpha=\alpha\sinh(k\Delta x)/\Delta x$,
which introduces a medium-derivative error into the equilibrium
balance. For $\Delta x=0.1,0.05,0.025,0.0125$, the ratio of the
PFD--NE and SLERP--SFE angular errors is
$2.004,2.001,2.0002,2.0001$, respectively. Thus, SLERP reduces the
angular error by approximately a factor of two in this
sampled-derivative test.

\begin{table}[!htb]
\centering\footnotesize

\setlength{\tabcolsep}{2pt}
\begin{minipage}[t]{0.485\textwidth}
\centering
\begin{tabular*}{\linewidth}{@{\extracolsep{\fill}}lrrr@{}}
\toprule
Method & Selected $\Delta x$ & $E_{\rm ang}$ (rad) & Time (s)\\
\midrule
\multicolumn{4}{@{}l@{}}{Analytic, $\varepsilon=-0.001$}\\
\addlinespace[2pt]
FD--NE & 0.025 & $8.879982\times10^{-5}$ & 0.04735\\
PFD--NE & 0.025 & $8.879982\times10^{-5}$ & 0.05046\\
SLERP--NE & 0.05 & $6.649442\times10^{-5}$ & 0.01733\\
SLERP--SFE & 0.05 & $6.649442\times10^{-5}$ & 0.02283\\
\addlinespace[5pt]
\multicolumn{4}{@{}l@{}}{Analytic, $\varepsilon=0$}\\
\addlinespace[2pt]
FD--NE & 0.025 & $6.804875\times10^{-5}$ & 0.05339\\
PFD--NE & 0.025 & $6.804875\times10^{-5}$ & 0.05609\\
SLERP--NE & 0.1 & $1.208209\times10^{-12}$ & 0.00556\\
SLERP--SFE & 0.1 & $1.207781\times10^{-12}$ & 0.00644\\
\bottomrule
\end{tabular*}
\end{minipage}\hfill
\begin{minipage}[t]{0.485\textwidth}
\centering
\begin{tabular*}{\linewidth}{@{\extracolsep{\fill}}lrrr@{}}
\toprule
Method & Selected $\Delta x$ & $E_{\rm ang}$ (rad) & Time (s)\\
\midrule
\multicolumn{4}{@{}l@{}}{Analytic, $\varepsilon=+0.001$}\\
\addlinespace[2pt]
FD--NE & 0.025 & $6.156770\times10^{-5}$ & 0.04991\\
PFD--NE & 0.025 & $6.156770\times10^{-5}$ & 0.05143\\
SLERP--NE & 0.05 & $6.639515\times10^{-5}$ & 0.01764\\
SLERP--SFE & 0.05 & $6.639515\times10^{-5}$ & 0.01946\\
\addlinespace[5pt]
\multicolumn{4}{@{}l@{}}{Sampled, $\varepsilon=0$}\\
\addlinespace[2pt]
FD--NE & 0.0125 & $3.455010\times10^{-5}$ & 0.28255\\
PFD--NE & 0.0125 & $3.455010\times10^{-5}$ & 0.29193\\
SLERP--NE & 0.025 & $6.803883\times10^{-5}$ & 0.07849\\
SLERP--SFE & 0.025 & $6.803883\times10^{-5}$ & 0.09183\\
\bottomrule
\end{tabular*}
\end{minipage}
\caption{Cost at an RMS angular-error target of $10^{-4}$ rad.
Each row selects the least median solver time among that method's
tested qualifying grids, so resolutions differ. Times are medians of
three repetitions, with the common step~\eqref{eq:ablation_timestep}
and residual tolerance $10^{-10}$. These are measured costs at a specified
accuracy, not same-grid speedups or certified optimal mesh choices.}
\label{tab:projection_cost}
\end{table}

For sampled $\nabla\alpha$, Table~\ref{tab:projection_cost}
shows that SLERP--SFE reaches $10^{-4}$ rad with $\Delta x=0.025$,
while FD--NE and PFD--NE require $\Delta x=0.0125$. Consequently,
SLERP--SFE is 3.08 and 3.18 times faster, respectively, and
SLERP--NE requires still less time. The savings come from reaching
the target accuracy on a coarser mesh, not from cheaper updates.

\FloatBarrier
\subsubsection{Time Update}
\label{sec:sfe_nonlinear_experiment}

We next compare SLERP--NE and SLERP--SFE for the relaxation
described in Proposition~\ref{prop:sfe_nonlinear_ordering}.
We consider the domain $[0,1]^2$ with $\alpha=e^{ky}$ and analytic
$\nabla\alpha$. The incoming orientation is prescribed as $(0,1)$
on $y=0$, and the original outgoing stencils are retained.
At all free nodes, the initial orientation is $(\sin e_0,\cos e_0)$.
The two methods use the same SLERP residual and a constant step given by
\begin{equation}\label{eq:sfe_relaxation_step}
 \Delta\tau=\frac{\nu\Delta x}{2e^k},\qquad 0<\nu<1,
\end{equation}
where $\nu$ denotes the CFL coefficient. For $\Delta x=1$ and $k>1$,
we choose $\nu=2/(k+1)$ to obtain the critical
step~\eqref{eq:sfe_cubic_constants} from the linearization,
independently of $e_0$. In the primary example, we take $k=3$,
$e_0=80^\circ$, $\nu=0.5$ and a residual tolerance of $10^{-10}$.
The $2\times2$ grid contains two free nodes and directly realizes
the one-cell nonlinear mechanism of the proposition.

The vector iteration agrees with the independent angular
recurrence~\eqref{eq:sfe_angular_maps} to within $3\times10^{-14}$
in vector norm. After three iterations, the angular errors for
SLERP--NE and SLERP--SFE are $2.298792\times10^{-3}$ and
$3.505346\times10^{-5}$ rad, respectively, a reduction by about
a factor of 66. The methods reach the common residual tolerance
in five and four iterations, respectively, and converge to the
same exact equilibrium.

\par}
\FloatBarrier

\FloatBarrier
{
\subsection{Numerical Convergence and Scalar-First Gradient Recovery}
\label{sec:scalar_recovery_comparison}

{
For comparison with the proposed orientation solver, we first solve
the scalar eikonal equation for the numerical traveltime $\phi_h$
and then recover its gradient by finite differences.
Specifically, we use first-order Godunov or Lax-Friedrichs (LF)
scalar solvers and apply either centered (CD) or upwind (UD)
differences for gradient recovery~\cite{benluozha10,leuqia0601}.
The four resulting scalar-first workflows are denoted by
Godunov-CD, Godunov-UD, LF-CD, and LF-UD. Let
\begin{equation}
\mathbf g_h = D\phi_h,
\qquad
\mathbf n_h^{S}
=
\frac{\mathbf g_h}{|\mathbf g_h|},
\qquad
\mathbf q_h^{S}
=
\alpha \mathbf n_h^{S},
\label{eq:scalar_magnitude_correction}
\end{equation}
whenever $\mathbf g_h\neq 0$, where $D$ denotes either centered
(CD) or upwind (UD) differencing. In this construction,
$\mathbf g_h$ is the gradient recovered directly from the
numerical traveltime, and normalization gives its direction
$\mathbf n_h^{S}$. We then obtain the reconstructed field
$\mathbf q_h^{S}$ by replacing the numerical gradient magnitude
with the prescribed slowness $\alpha$, while retaining the
recovered orientation. It follows that
$|\mathbf q_h^{S}-\alpha\mathbf n|
=
\alpha
|\mathbf n_h^{S}-\mathbf n|.$
This magnitude correction allows the directional errors of the scalar-first and direct orientation methods to be compared on the same basis.

We first examine the homogeneous point-source problem and then use the curved graded family to compare directional accuracy and computational cost at a prescribed accuracy.
}

\subsubsection{Homogeneous Point-Source Test}
\label{SubSec:2DConvergence}

We begin with the homogeneous eikonal equation $\|\nabla \phi(\mathbf{x})\| = 1$, subject to the point-source boundary condition $\phi(\mathbf{x}_0) = 0$, where $\mathbf{x}_0 = (0.5, 0.5)$.  
In this case, the vector field $\mathbf{n}$ represents the normalized gradient of $\phi$, corresponding to a unit direction field radiating outward from the source point $\mathbf{x}_0$.  
The governing equation for $\mathbf{n}$ follows from~\eqref{eq:eikonal-vec-nonconservative} with a constant slowness $\alpha(\mathbf{x}) = 1$. To prescribe the point-source boundary condition, we initialize $\mathbf{n}$ according to the known direction of the characteristic flow that emanates from the source.  
Specifically, for grid points in the neighborhood $0.48 \leq x, y \leq 0.52$, we set
$\mathbf{n}(x, y) = \frac{1}{r(x, y)} (x - x_0,\, y - y_0)$ where $r(\mathbf{x}) = \|\mathbf{x} - \mathbf{x}_0\|$. For all other grid points, the initialization is given by the analytical solution of the homogeneous eikonal equation, providing a smooth and physically consistent starting field for the iteration.  
These values are then iteratively updated by {the proposed orientation-equation method (SLERP--SFE)} until convergence.  
Simulations are performed on uniform grids with stepsize $dx = dy = 0.02,0.01,0.005,0.0025,0.00125$ and $0.000625$. $\epsilon$ is taken to be $10^{-10}$. The exact solution is given by $\tilde{\mathbf{n}}(x,y)
= \frac{1}{r(x,y)} (x - x_0,\; y - y_0)$.

To begin, we revisit the convergence behavior of conventional scalar fast sweeping methods for the eikonal equation.  
{The Godunov and Lax--Friedrichs solvers and their recovery labels are
defined at the start of Section~\ref{sec:scalar_recovery_comparison}.}
Their accuracy is assessed using the error metric
$E_{\text{traveltime},{\infty}}
    = \max_{\mathbf{x} \in \Omega}|\phi(\mathbf{x})-r(\mathbf{x})|$,
where $r(\mathbf{x}) = \|\mathbf{x}-\mathbf{x}_0\|$ denotes the exact solution for the homogeneous case, and the domain $\Omega$ is defined as the bounded region $[0,1]^2\setminus[0.48,0.52]^2$. The numerical results show that both schemes converge to the exact traveltime at approximately first order, consistent with standard observations for first-order scalar fast sweeping methods.

\begin{table}[!htb]
\centering
\caption{Error metrics used to assess the accuracy of the numerical gradient, slowness, and orientation fields. 
Here $\alpha$ and $\mathbf{n}$ denote the exact analytical slowness and direction fields, respectively. 
For each quantity, we report the discrete $L^1$-, $L^2$-, and $L^\infty$-norms of the corresponding scalar error. 
For instance, the slowness errors are defined as 
$E_{\text{slow},1} = \int |\|\nabla \phi\| - \alpha|$, 
$E_{\text{slow},2} = \left[\int |\|\nabla \phi\| - \alpha|^2 \right]^{1/2}$, 
and 
$E_{\text{slow},\infty} = \max |\|\nabla \phi\| - \alpha|$.}
\label{tab:error_metrics}
\begin{tabular}{c|c}
\textbf{Quantity} & \textbf{Error Definition} \\ \hline \\
{Slowness error}: $E_{\text{slow},1}$, $E_{\text{slow},2}$ and $E_{\text{slow},\infty}$ & $\left| \|\nabla \phi\| - \alpha \right|$ \\ 
\\
{Orientation error}: $E_{\text{ori},1}$, $E_{\text{ori},2}$ and $E_{\text{ori},\infty}$ &  $\left\|\frac{\nabla \phi}{\|\nabla \phi\|} - \mathbf{n}\right\|_2 $ \\
\\
{Gradient error}: $E_{\text{grad},1}$, $E_{\text{grad},2}$ and $E_{\text{grad},\infty}$ & $\|\nabla \phi - \alpha \mathbf{n}\|_2$ 
\end{tabular}
\end{table}

To assess the accuracy of the proposed method, we introduce three groups of error
metrics that evaluate complementary aspects of the numerical traveltime gradient.
The first group examines the accuracy of the gradient magnitude, verifying whether
$\|\nabla \phi\|$ matches the prescribed slowness $\alpha$, which reflects how well
the numerical solution adheres to the eikonal constraint.  
The second group isolates directional accuracy by comparing the normalized numerical
gradient with the orientation vector $\mathbf{n}$, thereby quantifying agreement in the
wavefront propagation direction.  
The third group measures how well the reconstructed field satisfies the fundamental
relation $\nabla \phi = \alpha \mathbf{n}$ by comparing the numerical gradient with the
product of the slowness and the computed orientation field.  
Together, these metrics provide a comprehensive evaluation of magnitude accuracy, directional fidelity and geometric consistency.  
The precise definitions are summarized in Table~\ref{tab:error_metrics}.

\begin{figure}[!htb]
\centering
\includegraphics[width=\textwidth]{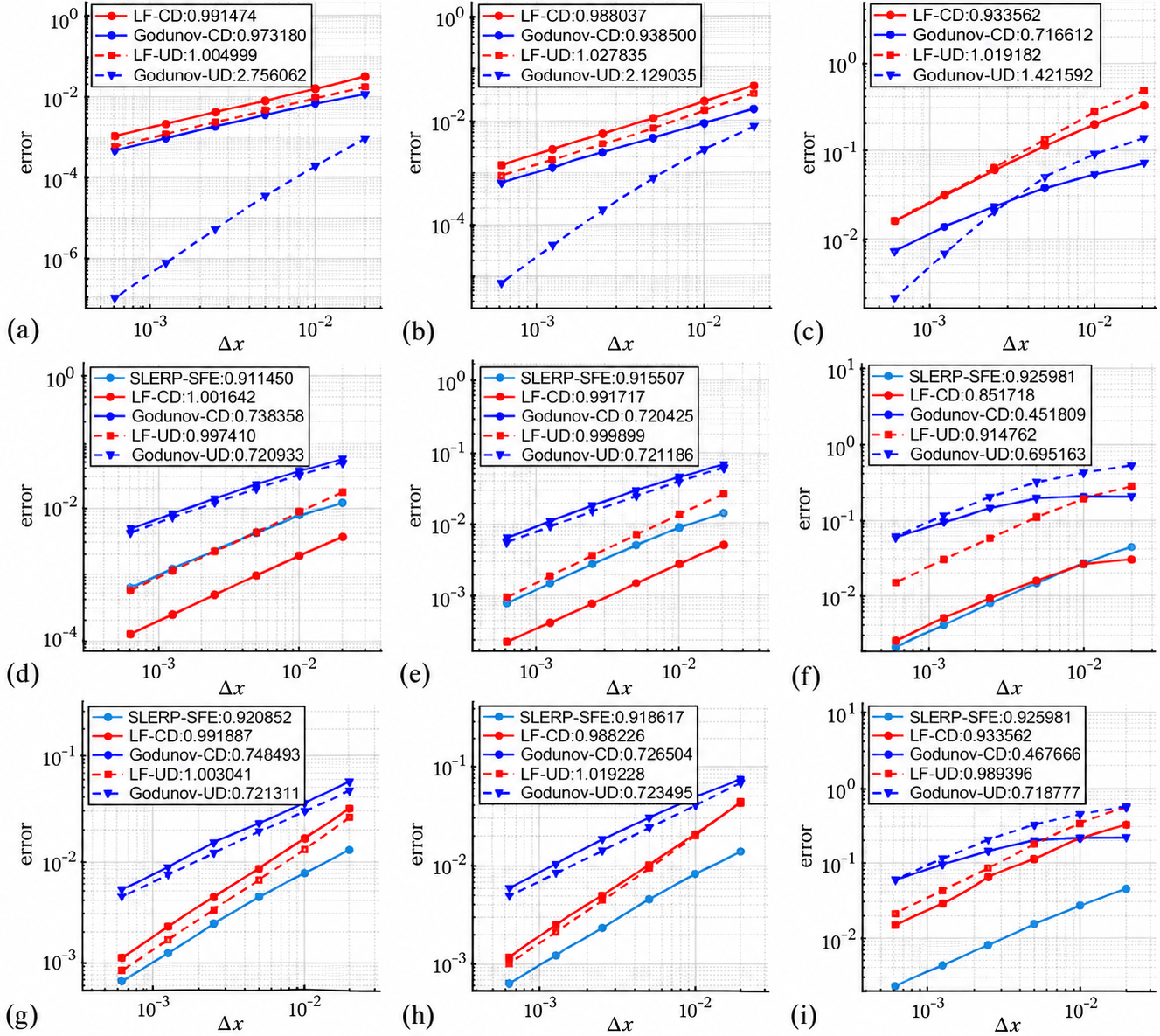}
\caption{Homogeneous two-dimensional point-source test of Section~\ref{SubSec:2DConvergence}.  The source is located at $\mathbf{x}_0=(0.5,0.5)$, the fixed source neighborhood is $[0.48,0.52]^2$, and the grid spacings are $dx=dy=0.02,0.01,0.005,0.0025,0.00125,0.000625$.  Panels (a)--(c) show the slowness errors $E_{\text{slow},1}$, $E_{\text{slow},2}$, and $E_{\text{slow},\infty}$; panels (d)--(f) show the orientation errors $E_{\text{ori},1}$, $E_{\text{ori},2}$, and $E_{\text{ori},\infty}$; panels (g)--(i) show the gradient-vector errors $E_{\text{grad},1}$, $E_{\text{grad},2}$, and $E_{\text{grad},\infty}$.  The reference orientation is $\tilde{\mathbf{n}}=(\mathbf{x}-\mathbf{x}_0)/\|\mathbf{x}-\mathbf{x}_0\|$. {The proposed orientation-equation method is denoted by SLERP--SFE, while Godunov--CD, Godunov--UD, LF--CD, and LF--UD denote the four scalar-eikonal workflows defined at the beginning of this section.}}
\label{fig:constant_slowness_1}
\end{figure}

Figure~\ref{fig:constant_slowness_1}(a--c) report the slowness errors measured by 
$E_{\text{slow},1}$, $E_{\text{slow},2}$, and $E_{\text{slow},\infty}$, where the traveltime fields are obtained using conventional scalar fast sweeping methods for the eikonal equation.
To evaluate the numerical slowness, we approximate the gradient of the computed traveltime using either an upwind or a central finite difference scheme{, corresponding to Godunov--CD, Godunov--UD, LF--CD, and LF--UD,} and then compute its Euclidean norm.
Ideally, this quantity should coincide with the prescribed slowness, as dictated by the eikonal relation $\|\nabla \phi\| = \alpha$.
Thus, these errors measure the residual violation of the eikonal constraint in the discrete solution.
The results show that the residual decreases at approximately first order under mesh refinement.
This behavior is consistent with the underlying first-order monotone discretization of the scalar eikonal equation.
We note that first-order fast sweeping schemes have been observed to exhibit superconvergence properties in smooth regions~\cite{benluozha10}, which can improve post-processed gradient accuracy in favorable cases.
In the present example, the reconstructed slowness retains essentially first-order accuracy, and the convergence behavior is nearly identical whether the gradient is approximated using first-order upwind differencing or second-order central differencing within the Lax--Friedrichs fast sweeping framework.
Among the scalar-based approaches used in this comparison, {Godunov--UD} achieves smaller residuals and exhibits slightly better performance.
Figure~\ref{fig:constant_slowness_1}(d--f) shows the orientation errors measured by 
$E_{\text{ori},1}$, $E_{\text{ori},2}$, and $E_{\text{ori},\infty}$.
All methods demonstrate first-order convergence.
{
In this test, Godunov--CD and Godunov--UD give larger orientation errors than the corresponding LF--CD and LF--UD workflows.
}
{SLERP--SFE preserves the prescribed slowness throughout the computation.}
Finally, Figure~\ref{fig:constant_slowness_1}(g--i) reports the errors in the full gradient vector $\nabla \phi$ measured in the $1$-, $2$-, and $\infty$-norms.

{
SLERP--SFE gives the smallest gradient errors before magnitude
correction, while LF--CD gives smaller $L^1$ and $L^2$ orientation
errors (Figure~\ref{fig:constant_slowness_1}(d,e)).
When magnitude correction is applied, the scalar workflows can
also give smaller gradient errors on sufficiently fine grids.
The results for this homogeneous point-source example therefore
support first-order convergence, although the proposed method
does not show a uniform advantage in directional accuracy.
}

\begin{figure}[!htb]
\centering
\includegraphics[width=\textwidth]{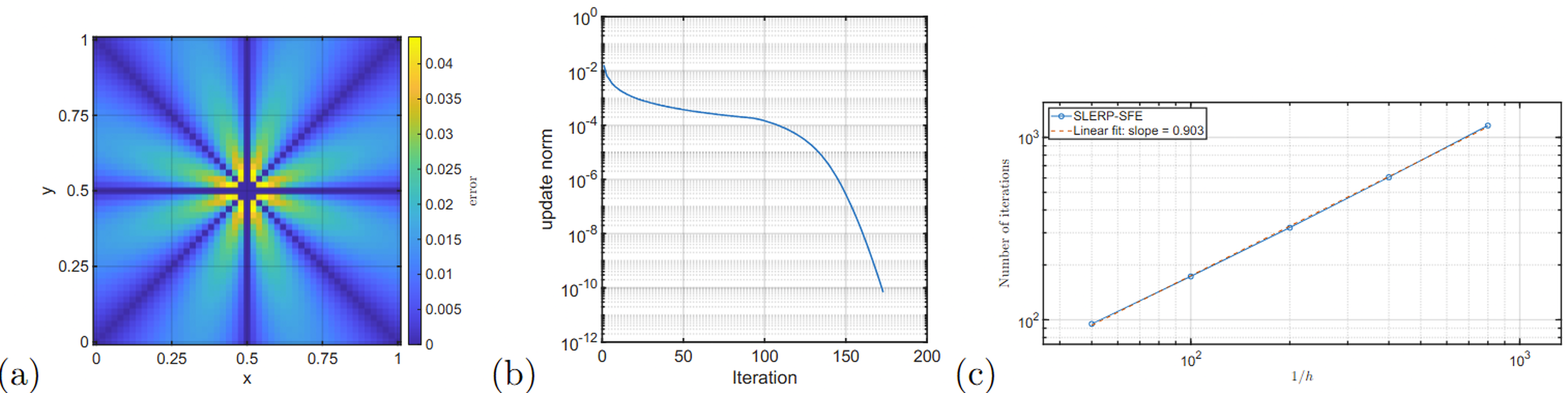}
\caption{Additional diagnostics for the homogeneous two-dimensional test.  Panel (a) shows the pointwise gradient error $E_{i,j}=\|\alpha(\mathbf{x}_{i,j})\mathbf{n}_{i,j}-\tilde{\mathbf{n}}_{i,j}\|_2$ on the computational mesh.  Panel (b) reports the nonlinear pseudo-time {update norm} $\max_{i,j}\|\mathbf{n}_{i,j}^{k+1}-\mathbf{n}_{i,j}^k\|_2$ for $dx=dy=0.01$.  Panel (c) reports the number of pseudo-time iterations required to reach the tolerance $\epsilon=10^{-10}$ for the grid refinements used in Figure~\ref{fig:constant_slowness_1}.}
\label{fig:experiment_constant_slowness}
\end{figure}

Figure~\ref{fig:experiment_constant_slowness} (a) shows the pointwise error in the computed orientation field evaluated by comparing the true gradient field $\grad\phi$ with its numerical approximation.
To recover the physical gradient from our vector-valued formulation, we scale the numerical orientation vector $\mathbf{n}$ by the slowness $\alpha(\mathbf{x})$.
Let $\mathbf{n}$ denote the numerical solution.
The pointwise error is defined as 
\begin{equation} \label{eq:error_on_mesh}
    E_{i,j} = \|\alpha(\mathbf{x}_{i,j})\mathbf{n}_{i,j} - \tilde{\mathbf{n}}_{i,j} \|_2
\end{equation}
As observed, the error is largest in the vicinity of the point source and decreases with distance.
This behavior is consistent with the singular nature of the solution at the source.
Figure~\ref{fig:experiment_constant_slowness}(b) illustrates the convergence behavior for $dx=0.01$.
The {update norm} $\|\mathbf{n}_{i,j}^{k+1} - \mathbf{n}_{i,j}^{k}\|_2$ decreases monotonically with iteration count in this experiment, demonstrating stable relaxation toward a steady state.
Unlike the classical scalar Godunov fast sweeping method, where the number of iterations required for convergence is often essentially independent of the mesh size, the present vector-valued pseudo-time iteration does not exhibit this behavior. 
As shown in Figure~\ref{fig:experiment_constant_slowness}(c), the number of iterations increases gradually as the grid is refined.
This behavior is likely attributable to the nonlinear coupling on the unit sphere introduced by our vector-valued formulation, which effectively increases the stiffness of the fixed-point iteration as resolution improves.
\par}

{
\subsubsection{Curved Inhomogeneous Fields and Cost}
\label{sec:scalar_curved_cost}

We now consider the graded family~\eqref{eq:graded_perturbation}
with $\varepsilon=0,\pm0.05$. The initial orientation is the same
nonexact constant field used in Section~\ref{sec:projection_experiments},
and we prescribe matching scalar traveltimes and orientations on
$x=0$ and $y=0$. In this example, both components of the exact
orientation remain positive. The characteristic argument
following~\eqref{eq:graded_2d} therefore identifies the smooth
field with the first-arrival solution.

The scalar Godunov and Lax--Friedrichs solvers are implemented
using their standard first-order formulas~\cite{zha05,kaooshqia04}.
For Lax--Friedrichs, we set the viscosities to $\sigma_x=\sigma_y=1$
and apply the prescribed extrapolation/max/min rule at the external
ghost nodes. To recover the gradient, CD uses centered differences
at interior nodes and second-order one-sided differences at the
edges. UD instead selects the neighbor with the lower traveltime
and uses the available side at an edge. All errors are evaluated
on the same unfixed physical nodes. Each solver is iterated until
its own maximum discrete residual reaches $10^{-10}$.
Since these residuals have different units, we do not interpret
equal tolerances as equal error estimates.

\begin{figure}[!htb]
\centering

\includegraphics[width=\textwidth]{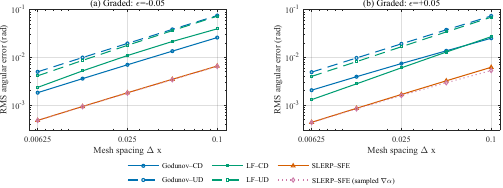}
\caption{Directional accuracy of scalar-first recovery and SLERP--SFE for the
curved graded fields. Panels (a,b) show RMS angular errors for
$\varepsilon=-0.05,+0.05$, respectively, on the common grids
$\Delta x=0.1,0.05,0.025,0.0125,0.00625$.
The two SLERP--SFE curves use analytic or sampled $\nabla\alpha$;
the sampled derivative uses CD and second-order one-sided edge differences,
without exact ghost slownesses. Both use step~\eqref{eq:ablation_timestep}.
The scalar labels denote the workflows defined above; magnitude correction
does not change their angular errors.}
\label{fig:scalar_recovery}
\end{figure}

The results in Figure~\ref{fig:scalar_recovery} show smaller
angular errors for SLERP--SFE. At $\Delta x=0.025$, its errors
are $1.806546\times10^{-3}$ and $1.663636\times10^{-3}$ rad,
respectively, improving on the best of the four scalar workflows
by factors of 3.86 and 3.62. With sampled $\nabla\alpha$, the
corresponding errors are $1.778644\times10^{-3}$ and
$1.583162\times10^{-3}$ rad. For $\varepsilon=0$, SLERP--SFE
with an analytic medium derivative preserves the curved
equilibrium to the stopping accuracy. With a sampled derivative,
its error decreases from $9.585\times10^{-4}$ to
$4.355\times10^{-6}$ rad over the same grids. These results
quantify the directional accuracy of the SLERP spatial
discretization for both choices of medium derivative.

\begin{table}[!htb]
\centering\small

\begin{tabular}{@{}rrrrrr@{}}
\toprule
 & \multicolumn{2}{c}{Godunov--CD} & \multicolumn{2}{c}{SLERP--SFE}
 & \shortstack{Time ratio\\scalar/SLERP--SFE}\\
$\varepsilon$ & $\Delta x$ & Time (s) & $\Delta x$ & Time (s) & \\
\midrule
0 & 0.003125 & 0.01449 & 0.1 & 0.00694 & 2.086\\
-0.05 & 0.003125 & 0.03704 & 0.0125 & 0.68053 & 0.054\\
0.05 & 0.0015625 & 0.14707 & 0.0125 & 0.67808 & 0.217\\
\bottomrule
\end{tabular}
\caption{Cost at the common RMS angular-error target $10^{-3}$ rad.
Each workflow uses its least-cost qualifying tested grid, selected from
three-run medians; scalar grids extend to $\Delta x=0.0015625$ as needed.
Godunov--CD is the cheapest of the four scalar workflows tested.
Selected workflows are remeasured in seven paired batches after three
warmup solves. Batch means use 1--9 calls according to the pilot cost;
the entries are their medians. SLERP--SFE here uses analytic $\nabla\alpha$.
The scalar/SLERP--SFE ratio exceeds one only for $\varepsilon=0$;
this ordering holds in every batch, also with sampled $\nabla\alpha$.}
\label{tab:scalar_target_cost}
\end{table}

Table~\ref{tab:scalar_target_cost} compares the computational
costs at the prescribed angular accuracy. SLERP--SFE has a cost
advantage for the preserved equilibrium, while Godunov is cheaper
for both perturbations, also with sampled $\nabla\alpha$.
The reported times include initialization, monitoring, gradient
recovery, normalization and sampled medium differentiation.
They exclude input construction, diagnostics and I/O. 

\par}
\FloatBarrier

\subsection{Waveguide Model}
\label{SubSec:Waveguide}

We consider the waveguide model introduced in~\cite{qialeu04}, where the slowness is given by
$$\alpha(x, y) = \left[\, 1.1 - \exp(-0.5x^2) \,\right]^{-1}$$ for $(x,y)\in[-1,1]\times[0,0.7]$. In this setting, rays enter the domain from the lower boundary $y=0$ and propagate predominantly upward.  
Accordingly, we prescribe the incoming boundary condition $\mathbf{n}(x,0) = (0,1)$.
The initial condition is constructed using the solution of the corresponding homogeneous problem under the same boundary condition; that is, we set $\mathbf{n}(x,y)=(0,1)$ at all interior grid points.   
Simulations are performed on uniform grids with the stepsize 
$dx = dy = 0.05, 0.025, 0.0125$ and $0.00625$. 
We set $\epsilon=10^{-10}$.
Since an analytical solution is not available, we do not compute the error in this case. 
Instead, we compare the computed gradient vector field with rays traced from the source.

\begin{figure}[!htb]
\centering
\includegraphics[width=\textwidth]{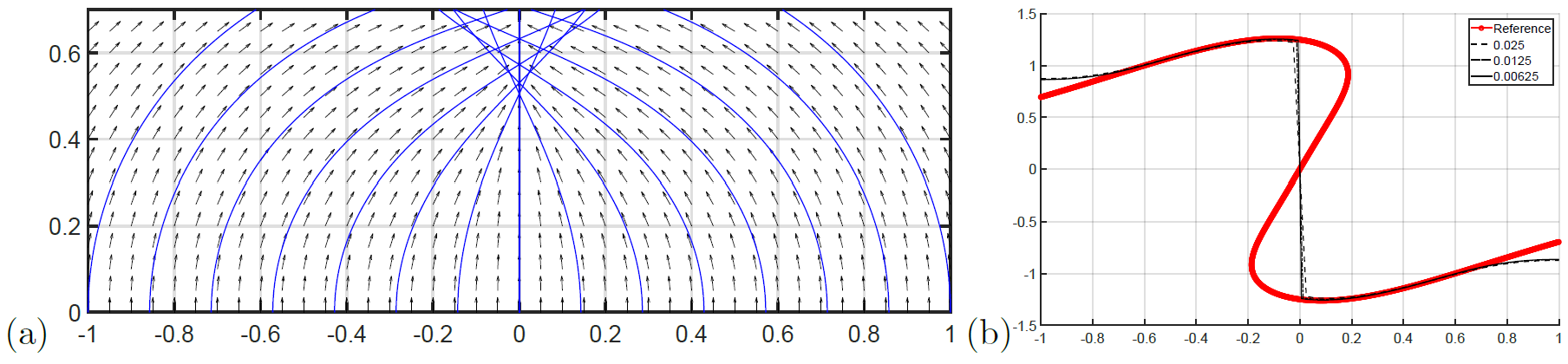}
\caption{Waveguide model of Section~\ref{SubSec:Waveguide} with $\alpha(x,y)=[1.1-\exp(-0.5x^2)]^{-1}$ on $[-1,1]\times[0,0.7]$ and incoming boundary condition $\mathbf{n}(x,0)=(0,1)$.  Panel (a) shows the computed orientation field together with reference rays.  Panel (b) compares the arrival angle $\theta$ along the top boundary $y=0.7$ for grid spacings $dx=dy=0.05,0.025,0.0125,0.00625$ against the ray-tracing reference.}
\label{fig:waveguide}
\end{figure}

Figure~\ref{fig:waveguide} shows representative numerical results. 
Figure~\ref{fig:waveguide}(a) presents the computed orientation field, which illustrates the smooth bending of propagation directions caused by the waveguide structure.
To compare with the expected physical behavior, we overlay several rays on the mesh.
These rays serve as reference solutions, obtained using the ray tracing method in~\cite{qialeu04} on a sufficiently fine mesh.
Figure~\ref{fig:waveguide}(b) plots the arrival angle $\theta$ along the top boundary ($y=0.7$). 
We note that regions not covered by the reference rays are less relevant, as they depend on information outside the computational domain.
Moreover, since one-sided differences are used at the left and right boundaries, the vectors in these regions are not updated according to the true characteristic flow. 
For this reason, we restrict our comparison to the region enclosed by the blue rays. 
Within this region, we observe that the computed arrival angles agree well with the reference solution.
Overall, these results show that the structure-preserving scheme can accurately capture the directional behavior governed by the spatially varying slowness.

\subsection{Sinusoidal Model}
\label{SubSec:ExSine}

\begin{figure}[!htb]
\centering
\includegraphics[width=\textwidth]{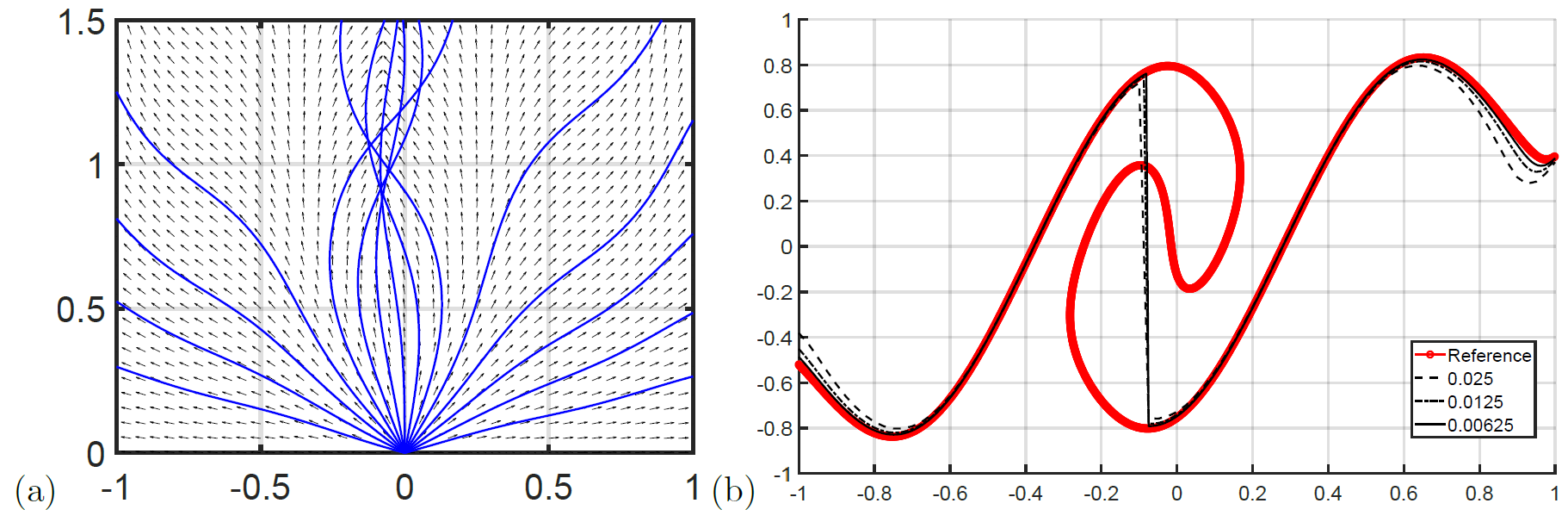}
\caption{Sinusoidal slowness model of Section~\ref{SubSec:ExSine}.  Panel (a) shows the computed orientation field and representative rays initialized near the point source.  Panel (b) compares the propagation angle $\theta$ along the top boundary $y=1.5$ for the grid refinements used in the experiment against the ray-tracing reference solution.}
\label{fig:sinusoidal}
\end{figure}

We further assess the performance of the proposed method using the sinusoidal model introduced in~\cite{sym98,symqia03slowm,qialeu04}, where the slowness field is given by $\alpha(x,y) = (1 + 0.2\,\sin(0.5\pi y)\sin\!\big[3\pi(x+0.55)\big])^{-1}$, where $(x,y)\in[-1,1]\times[0,1.5]$.
This spatially oscillatory coefficient breaks the symmetry present in previous examples and induces highly anisotropic propagation effects.  
As a result, the exact wavefronts exhibit directional bending and curvature variations that depend sensitively on location. 
This experiment provides additional validation of our scheme, demonstrating its ability to capture complex, non-symmetric propagation patterns in a heterogeneous medium. 
We consider a point source located at the origin. 
Because characteristics emanate outward from this location, we prescribe the direction field on the small neighborhood $-0.06 \leq x \leq 0.06$, $0 \leq y \leq 0.06$.
As in previous examples, this boundary specification is imposed in a fixed region to avoid the singularity at the source. 
For the remaining grid points, we initialize by employing the solution to the homogeneous system with the same boundary condition. Simulations are carried out on uniform grids with stepsize $dx = dy = 0.025, 0.0125$ and $0.00625$. 
The parameter in the stopping criteria $\epsilon$ is taken to be $10^{-10}$.
As in the previous experiment, we do not compute the error since the analytical solution is unavailable. 
Instead, we compare the computed gradient vector field with rays traced from the source.

The numerical results are shown in Figure~\ref{fig:sinusoidal}.
Figure~\ref{fig:sinusoidal}(a) displays the computed orientation field $\mathbf{n}$, where the arrows indicate the gradient directions obtained from our model, and the blue curves represent a set of rays traced from the point-source boundary.
These rays serve as reference solutions computed using the ray tracing method in~\cite{qialeu04}.
Since the characteristics originate from the source, the rays are initialized at the origin. 
As shown in the figure, the ray paths follow the underlying velocity field closely, indicating good qualitative agreement between the computed characteristic directions and the vector field produced by the proposed method.
Figure~\ref{fig:sinusoidal}(b) shows the arrival angle along the top boundary $(y=1.5)$ for both the proposed model and the reference solution.
The results from our model agree very well with the reference, and the agreement improves as the grid is refined. 
This convergence demonstrates both the accuracy and the geometric consistency of the proposed approach.

\subsection{Three-Dimensional Examples}

While the preceding discussions focus on two-dimensional cases, the proposed structure-preserving method naturally extends to three dimensions without modification. 
This section presents several three-dimensional examples to demonstrate the extension and convergence behavior of the proposed scheme.

\subsubsection{Radial Example}
\label{SubSec:Ex3DRadial}

\begin{figure}[!htb]
\centering
\includegraphics[width=\textwidth]{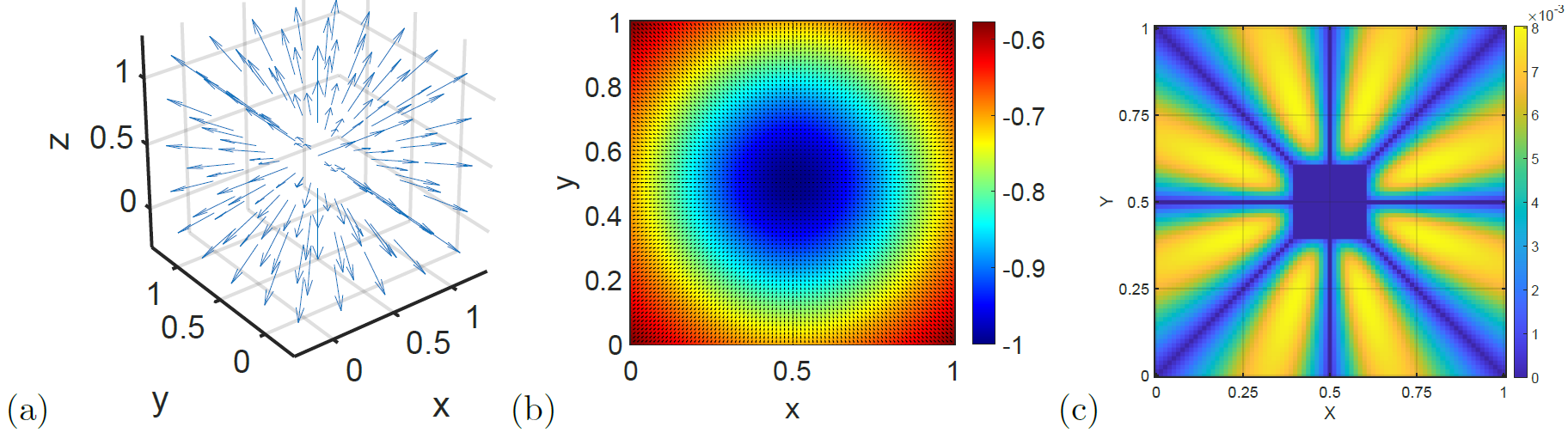}
\caption{Three-dimensional radial test of Section~\ref{SubSec:Ex3DRadial} for $\|\nabla\phi\|=1+r(\mathbf{x})^2$ on $[0,1]^3$ with source $\mathbf{x}_0=(0.5,0.5,0.5)$.  Panel (a) shows representative computed orientation vectors in three dimensions.  Panel (b) shows the vector field on the cross section $z=0$.  Panel (c) shows the pointwise gradient error on the cross section $z=0.49$, excluding the fixed source neighborhood.}
\label{fig:radial_dim3_1}
\end{figure}

\begin{figure}[!htb]
\centering
\includegraphics[width=\textwidth]{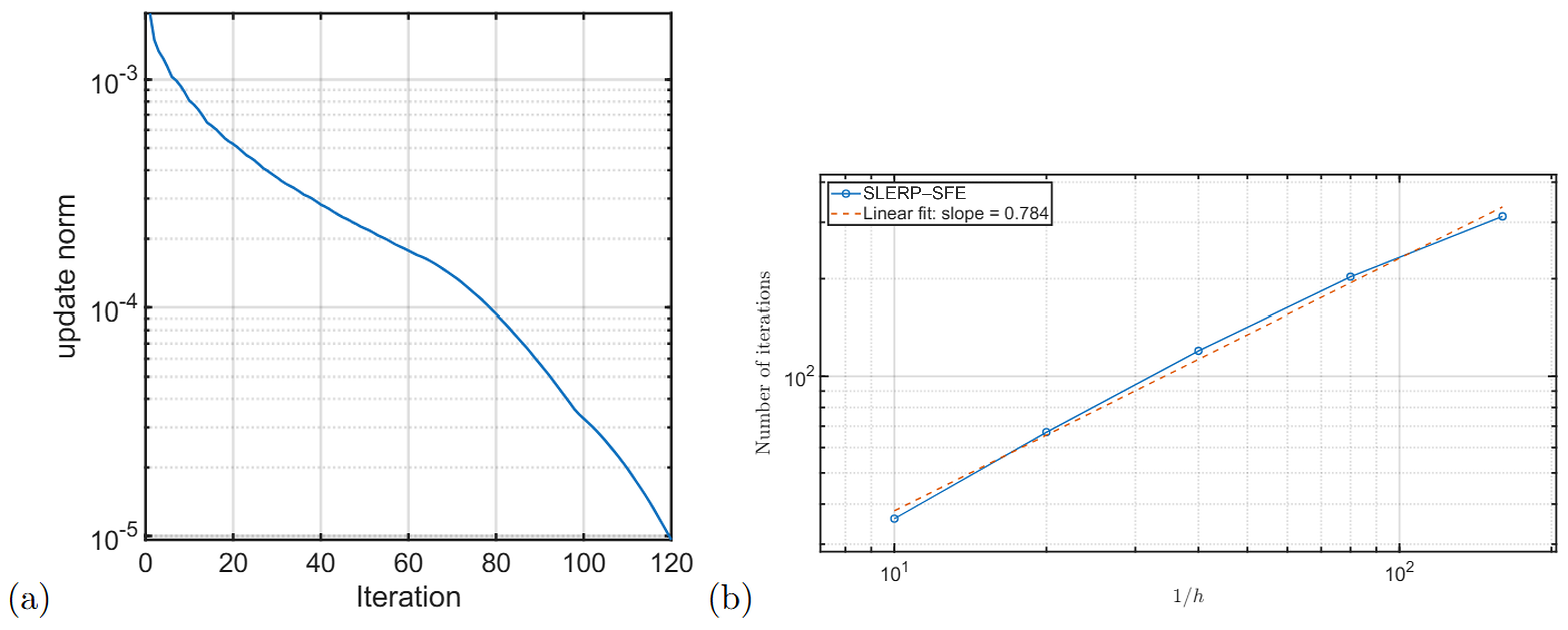}
\caption{Pseudo-time convergence diagnostics for the three-dimensional radial test.  Panel (a) shows the decay of $\max_{i,j,k}\|\mathbf{n}_{i,j,k}^{m+1}-\mathbf{n}_{i,j,k}^{m}\|_2$ for $dx=dy=dz=0.025$.  Panel (b) reports the number of pseudo-time iterations needed to reach the stopping tolerance for the grid spacings $0.1,0.05,0.025,0.0125,0.00625$.}
\label{fig:experiment_2_2}
\end{figure}

We consider $\|\nabla \phi(\mathbf{x})\| = 1 + r(\mathbf{x})^2$ in $\mathbf{x} \in [0,1]^3$ where $r(\mathbf{x}) = \|\mathbf{x}-\mathbf{x}_0\|$. The boundary condition is given by $\phi(\mathbf{x}_0) = 0$ at $\mathbf{x}_0 = (0.5,0.5,0.5)$. Numerically, for grid points in the neighborhood $0.39 \leq x,y,z \leq 0.61$, we prescribe
$\mathbf{n}(x,y,z) = \frac{(x-x_0,\; y-y_0,\; z-z_0)}{r(x,y,z)}$, thereby enforcing the known characteristic directions in an $O(1)$ region around the source.  
The remaining interior points are initialized using the solution to the corresponding homogeneous problem.
Simulations are performed on uniform grids with stepsize $dx=dy=dz = 0.1,0.05,0.025,0.0125$ and $0.00625$.
The convergence tolerance is set to $\epsilon=10^{-5}$.

Figure~\ref{fig:radial_dim3_1}(a) presents a three-dimensional visualization of the computed gradient field, where representative vectors are displayed at selected grid points. The numerical solution clearly exhibits vectors radiating outward from the point source, in full agreement with the analytical structure of the radial model. To provide further insight, Figure~\ref{fig:radial_dim3_1}(b) shows a two-dimensional slice at $z=0$. On this plane, the $x$-- and $y$--components determine the in-plane direction of each vector, while the $z$--component is encoded through color intensity. The resulting field again demonstrates outward propagation from the source, with the magnitude increasing toward the center, consistent with the radial symmetry of the full three-dimensional solution. Figure~\ref{fig:radial_dim3_1}(c) displays the error distribution on the slice $z=0.49$, computed using the metric defined in~\eqref{eq:error_on_mesh}. As in the two-dimensional experiment, the fixed central region $0.39 \leq x,y \leq 0.61$ is excluded from the error calculation. Because this slice lies close to the source, the geometry is effectively two-dimensional, and the qualitative error pattern resembles that observed previously. The largest errors occur near the outer boundaries of the domain, where characteristic surfaces intersect the domain edges and directional discontinuities naturally develop.

Figure~\ref{fig:experiment_2_2}(a) illustrates the monotone decay of the {update norm} throughout the iterative process for this test, providing numerical evidence of convergence for the proposed structure-preserving method. The corresponding iteration count is shown in Figure~\ref{fig:experiment_2_2}(b), where, consistent with the two-dimensional case, the number of iterations increases gradually as the grid is refined.

\begin{figure}[!htb]
\centering
\includegraphics[width=\textwidth]{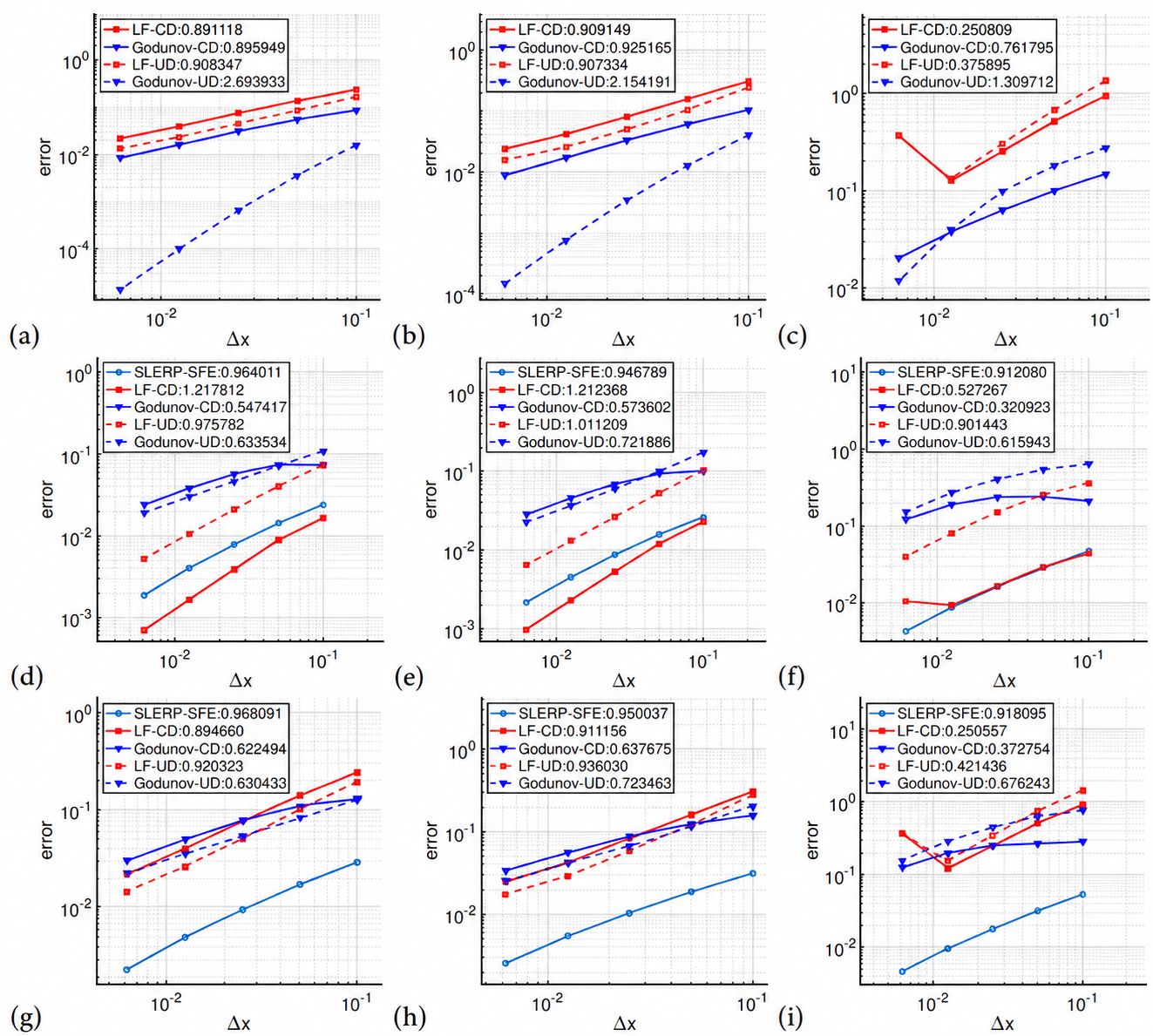}
\caption{Quantitative convergence for the three-dimensional radial test.  Panels (a)--(c) report the slowness errors $E_{\text{slow},1}$, $E_{\text{slow},2}$, and $E_{\text{slow},\infty}$; panels (d)--(f) report the orientation errors $E_{\text{ori},1}$, $E_{\text{ori},2}$, and $E_{\text{ori},\infty}$; panels (g)--(i) report the gradient-vector errors $E_{\text{grad},1}$, $E_{\text{grad},2}$, and $E_{\text{grad},\infty}$. {SLERP--SFE is compared with the four scalar-based workflows Godunov--CD, Godunov--UD, LF--CD, and LF--UD on the same sequence
of grid spacings.}}
\label{fig:experiment_2_3}
\end{figure}

Finally, Figures~\ref{fig:experiment_2_3} summarize the quantitative convergence results using the error metrics defined in Table~\ref{tab:error_metrics}.  Figures~\ref{fig:experiment_2_3}(a-c) report the corresponding slowness errors, while Figures~\ref{fig:experiment_2_3}(d-f) compare the orientation errors across all methods.  
Figures~\ref{fig:experiment_2_3}(g-i) further present the errors in the reconstructed gradient field $\nabla \phi$. {In panels (g--i), we observe approximately first-order convergence for
SLERP--SFE, with smaller errors than those obtained from the uncorrected
scalar gradients. However, this improvement does not hold uniformly for
the direction field, since LF--CD gives smaller $L^1$ and $L^2$
orientation errors in panels (d,e).}

\subsubsection{Vinje's Gaussian Model}
\label{SubSec:ExVinje}

We further assess the performance of the proposed numerical scheme using Vinje's Gaussian velocity model.
This model, introduced in \cite{viag96}, prescribes the slowness field 
\[
\alpha(x,y,z) = \left[3-1.75\exp\left(- \frac{x^2+y^2+(z-0.75)^2}{0.5^2} \right)\right]^{-1}
\]
, for $(x,y,z)\in[-1,1]^3$. We place a point source at the origin.
To construct the boundary condition required in our formulation, we treat grid points inside the small cube $-0.05 \leq x,y,z \leq 0.05$ as source-region points. Simulations are performed on uniform grids with the stepsize $dx=dy=dz=0.02$.
We set the convergence tolerance to $\epsilon=10^{-8}$.

\begin{figure}[!htb]
\centering
\includegraphics[width=\textwidth]{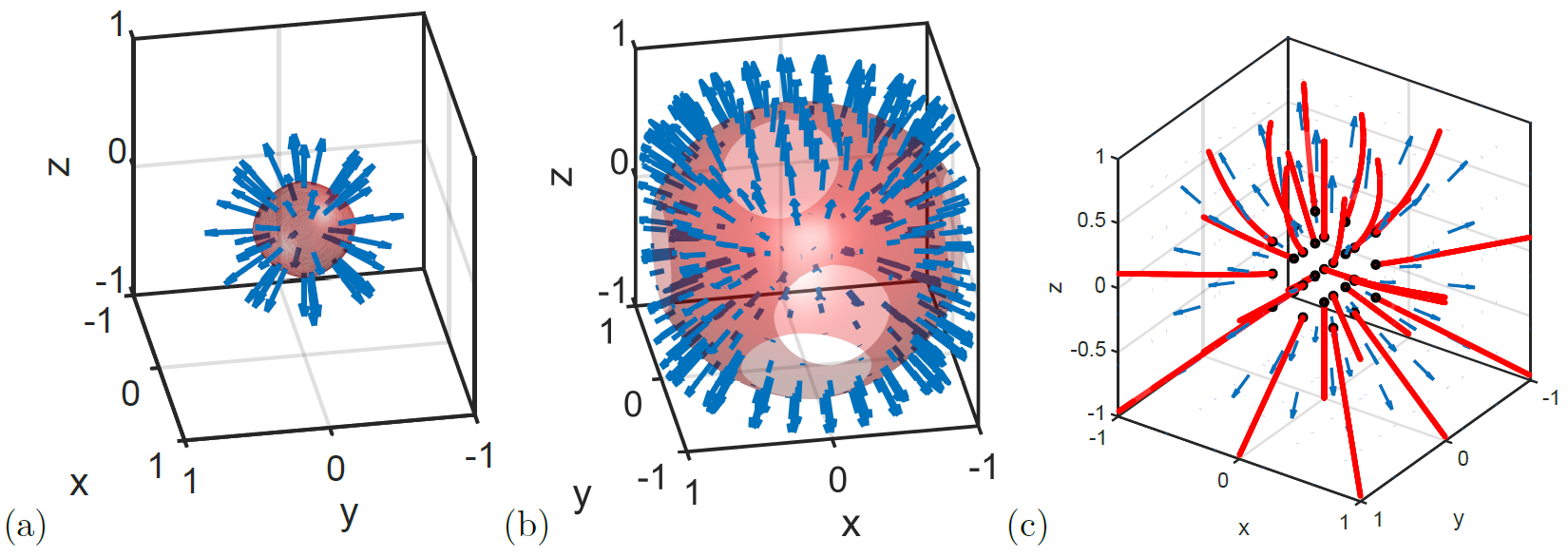}
\caption{Vinje's Gaussian model of Section~\ref{SubSec:ExVinje}.  Panels (a) and (b) show the computed gradient field plotted on the level surfaces $\phi=0.12$ and $\phi=0.37$, respectively.  Panel (c) shows a downsampled orientation field with rays traced from points near the origin, illustrating consistency between the computed direction field and the ray trajectories.}
\label{fig:Vinje_Model}
\end{figure}

The numerical results appear in Figure~\ref{fig:Vinje_Model}. Figures~\ref{fig:Vinje_Model}(a) and \ref{fig:Vinje_Model}(b) show the gradient vectors restricted to the level surfaces $\phi=0.12$ and $\phi=0.37$, respectively.
In both cases, the computed vectors are orthogonal to the corresponding level sets.
This behavior agrees with the theoretical property $\mathbf{n} = \grad\phi/\|\grad\phi\|$.
This observation further supports the geometric fidelity of the proposed discretization. Figure~\ref{fig:Vinje_Model}(c) shows a downsampled visualization of the computed gradient vector field inside the domain.
To illustrate the effect of the spatially varying slowness, we emit several rays from points near the origin.
We compute their trajectories by integrating the system $d\mathbf{x}/dt=\mathbf{n}$.
The resulting curves are consistent with the characteristic paths expected from the computed orientation field.

{
\subsubsection{Projection Comparison in Three Dimensions}
\label{sec:projection_3d_comparison}

We next compare FD--NE, PFD--NE, SLERP--NE, and SLERP--SFE in three
dimensions. For this comparison, we use the same pseudo-time step and
angular-error measure as in Section~\ref{sec:projection_experiments},
with the tolerance on the maximum tangential residual set to $10^{-8}$.
The timing protocol is also kept unchanged.

To examine the effect of rotating the equilibrium and introducing
nonplanar perturbations, we consider a family of fields on $[0,1]^3$.
Specifically, we choose $\mathbf a=(1,1,0)/\sqrt2$,
$\mathbf b=(-1,1,2)/\sqrt6$, $\mathbf c=\mathbf a\times\mathbf b$,
and define
\begin{equation}\label{eq:geodesic_3d_perturbed}
 \phi_\varepsilon(\mathbf x)=e^{0.6v}
 \left[\frac{\sin(0.15+0.6u)}{0.6}
       +\frac{\varepsilon\sin(\pi w)}{\pi}\right],
 \quad (u,v,w)=(\mathbf a\cdot\mathbf x,\mathbf b\cdot\mathbf x,
                         \mathbf c\cdot\mathbf x),
\end{equation}
where $\alpha=\|\nabla\phi_\varepsilon\|$ and $\nabla\alpha$ is
evaluated analytically. On the boundary faces where the exact field
points into the domain, we prescribe the exact incoming orientations
and keep these values fixed. At all remaining nodes, the initial
orientation is given by $\cos(0.15)\mathbf a+\sin(0.15)\mathbf b$.
Figure~\ref{fig:3d_orientation_fields}(a,b) presents the computed
fields for $\varepsilon=0,0.05$, while panel (c) compares the
out-of-plane angles for $\varepsilon=0,\pm0.05$. All computations
use the same grid with $\Delta x=0.025$. The unperturbed field lies in
$\operatorname{span}(\mathbf a,\mathbf b)$, whereas both perturbations
produce directions outside this plane. To describe this deviation,
panel (c) plots
$\arcsin[(\mathbf c\cdot\mathbf n_\varepsilon^{\rm prop})/\|\mathbf n_\varepsilon^{\rm prop}\|]$
along $x=z=0.5$. The normalization in this expression is used only
for the diagnostic.

\begin{figure}[!htb]
\centering

\includegraphics[width=\textwidth]{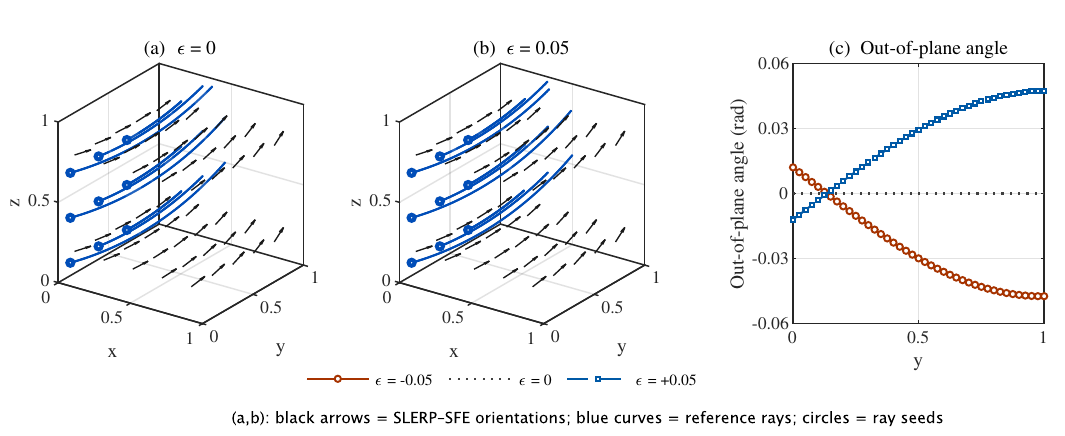}
\caption{Computed orientations and signed out-of-plane angles for the
rotated three-dimensional family~\eqref{eq:geodesic_3d_perturbed}.
All three computations use the same $\Delta x=0.025$ grid, the proposed
method, analytic $\nabla\alpha$, and residual tolerance $10^{-8}$.
(a,b) Fields for $\varepsilon=0,0.05$, respectively. Black arrows of a common length
show the numerical orientations at selected nodes on $z=0.2$ and $z=0.8$;
blue curves integrate the exact orientation field from the circles on
$x=0$ to the outgoing boundary. The axes use scaled coordinates and an
orthographic view. The rays are independent references; the arrows are
sampled numerical solutions.
(c) Signed angle of the computed orientation out of
$\operatorname{span}(\mathbf a,\mathbf b)$ along $x=z=0.5$, in radians.
Circles and squares mark all $41$ nodes for $\varepsilon=-0.05$ and
$+0.05$, respectively, with straight segments joining them. The dotted
curve corresponds to $\varepsilon=0$. All curves use the same symmetric
angle scale; no angle, arrow, or ray displacement is magnified.}
\label{fig:3d_orientation_fields}
\end{figure}

On the grid with $\Delta x=0.025$, the angular error of SLERP--SFE is
$E_{\rm ang}=4.20\times10^{-12}$ rad for $\varepsilon=0$. This result
confirms that the equilibrium is preserved after rotation relative
to the grid, as stated in Proposition~\ref{prop:geodesic_equilibrium};
the directions in this case remain on one great circle. For the
nonplanar field with $\varepsilon=0.05$, the errors obtained using
PFD--NE and SLERP--SFE are very close, with values of
$9.598842\times10^{-4}$ and $9.608098\times10^{-4}$ rad, respectively.

{
All four methods preserve unit length up to round-off error, with a maximum defect of $1.89\times10^{-14}$. This algebraic constraint does not, however, guarantee scalar-gradient compatibility in three dimensions. We therefore examine the curl of the reconstructed field for the nonplanar case ($\epsilon=0.05$).
}

\begin{table}[!htb]
\centering\small

\begin{tabular}{@{}rrrr@{}}
\toprule
$\Delta x$ & $\|C_h\mathbf q_h\|_{\rm RMS}$ &
$\|C_h\mathbf q_*\|_{\rm RMS}$ &
$\|C_h(\mathbf q_h-\mathbf q_*)\|_{\rm RMS}$\\
\midrule
0.1 & $1.21259\times10^{-2}$ & $6.60940\times10^{-4}$ & $1.18164\times10^{-2}$\\
0.05 & $7.43358\times10^{-3}$ & $1.63305\times10^{-4}$ & $7.36958\times10^{-3}$\\
0.025 & $4.11899\times10^{-3}$ & $4.05385\times10^{-5}$ & $4.10445\times10^{-3}$\\
\bottomrule
\end{tabular}
\caption{Curl diagnostic for SLERP--SFE in the nonplanar three-dimensional
field at $\varepsilon=0.05$, using analytic $\nabla\alpha$ and residual
tolerance $10^{-8}$. All columns use the same interior stencils and
scaled coordinates. Nonzero curl of the sampled exact gradient is
finite-difference truncation error.}
\label{tab:curl_refinement}
\end{table}

{
The results in Table~\ref{tab:curl_refinement} are obtained by applying
the same centered curl operator $C_h$ to the reconstructed field
$\mathbf q_h=\alpha\mathbf n_h$ and the sampled exact gradient
$\mathbf q_*$ for~\eqref{eq:geodesic_3d_perturbed}, $\varepsilon=0.05$.
We compute the RMS norms at interior nodes whose coordinate neighbors
are also unfixed. The column for the exact field measures the error
introduced by the stencil, while the last column reports
$C_h(\mathbf q_h-\mathbf q_*)$. This comparison separates the
contribution of the numerical field from the truncation error of
the curl stencil.
}

\par}

{
\subsection{Observed Computational Scaling}
\label{sec:runtime_scaling}

We further examine how the computational cost changes under grid
refinement, using the graded two-dimensional and nonplanar
three-dimensional fields with $\varepsilon=0.05$. The initial and
boundary data, analytic $\nabla\alpha$, pseudo-time step and residual
tolerances are kept as in the preceding tests.
Figure~\ref{fig:runtime_scaling} compares FD--NE, PFD--NE, SLERP--NE, and
SLERP--SFE, together with Godunov--CD and LF--CD in two dimensions.
For each workflow, the measured time includes initialization and
convergence monitoring, as well as scalar recovery when applicable.
The construction of the input data, diagnostics and I/O are excluded.

\begin{figure}[!htb]
\centering

\includegraphics[width=\textwidth]{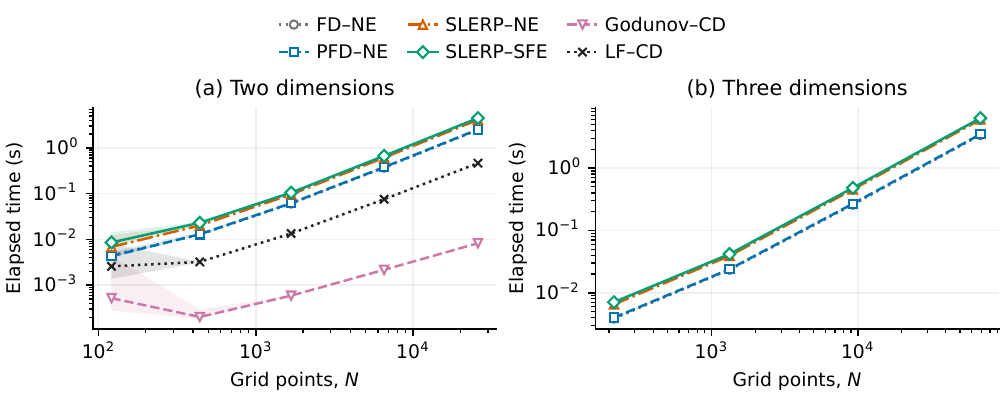}
\caption{Elapsed time $T$ against the total number $N$ of physical grid points.
(a) $\Delta x=0.1,0.05,0.025,0.0125,0.00625$, tolerance $10^{-10}$.
(b) $\Delta x=0.2,0.1,0.05,0.025$, tolerance $10^{-8}$.
Both use $\varepsilon=0.05$ and one MATLAB computational thread.
Markers are medians of five paired batch means after two warmup solves;
shading spans their minimum and maximum. Batches average 1--32 solves
according to pilot duration. The FD--NE and PFD--NE curves nearly coincide, as do the SLERP--NE and SLERP--SFE curves.
The reported slopes use the finest three grids in each panel.}
\label{fig:runtime_scaling}
\end{figure}

Over the tested grid refinements, SLERP--SFE requires 213--1619
iterations in two dimensions and 137--601 in three dimensions.
Fitting the results against $N$ on the finest three grids gives
exponents of 1.37 and 1.27 for the total computational time,
0.40 and 0.26 for the iteration count, and 0.97 and 1.01 for the
time per iteration, respectively. These results indicate that the
work per iteration is close to linear over the tested range, while
the growth in the iteration count accounts for the superlinear
total cost. SLERP--NE requires the same number of iterations on
these grids. In the two-dimensional comparison, Godunov requires
one cycle of four directional sweeps on each grid, with a measured
time exponent of 0.96. LF requires 15--74 cycles, with a
corresponding exponent of 1.29.
The comparisons in Tables~\ref{tab:projection_cost}
and~\ref{tab:scalar_target_cost} assess the computational cost at
prescribed angular accuracy.

\FloatBarrier
\par}

\section{Conclusion}
\label{sec:conc}

{We have developed a first-order, structure-preserving method for
directly computing eikonal traveltime orientations in smooth first-arrival
regions. The spatial discretization uses SLERP endpoint derivatives
to construct the upwind approximation in the tangent space, while
the spherical forward Euler update keeps every iterate on the unit
sphere. Consequently, the reconstructed field satisfies
$\|\alpha\mathbf n_h\|=\alpha$ up to round-off error, so that the
prescribed slowness is preserved throughout the computation.}

{The SLERP spatial discretization exactly preserves a family of curved
orientation equilibria. In the nearby perturbation tests, the removal
of background discretization bias improves the accuracy of the
absolute orientation and allows selected error targets to be reached
on coarser grids, with lower computational costs than FD--NE/PFD--NE.
This property is shared by SLERP--NE and SLERP--SFE, since it arises
from their spatial discretization. When compared with the four
scalar-first workflows on the same grid, SLERP--SFE also produces
more accurate directions in the tested curved media, including
the cases where the medium derivatives are obtained from sampled
slowness.}

{The one-cell analysis also establishes a conditional angular-error
advantage of SFE, which is separate from the spatial preservation
property. Together with the two- and three-dimensional comparisons,
this analysis clarifies the respective contributions of geometric
spatial reconstruction and the pseudo-time update. These results
provide a basis for developing higher-order orientation solvers.}

{Future extensions include SIDER or SENO spherical
reconstruction~\cite{fonleu23} and STVDRK pseudo-time
integration~\cite{leuchalee24,leeleu26}. We also plan to study convergence
of the nonlinear relaxation and applications to adjoint traveltime
tomography and related inverse problems.}

\bibliographystyle{plain}
\bibliography{reference,syleung}

\begin{thebibliography}{10}

\bibitem{bartels2009convergent}
S{\"o}ren Bartels, Christian Lubich, and Andreas Prohl.
\newblock Convergent discretization of heat and wave map flows to spheres using
  approximate discrete lagrange multipliers.
\newblock {\em Mathematics of computation}, 78(267):1269--1292, 2009.

\bibitem{benluozha10}
J.~D. Benamou, S.~Luo, and H.~Zhao.
\newblock {A Compact Upwind Second Order Scheme for the Eikonal Equation}.
\newblock {\em J. Computational Mathematics}, 28(4):489--516, 2010.

\bibitem{bor16}
A.~Boritchev.
\newblock {Multidimensional Potential Burgers Turbulence}.
\newblock {\em Commun. Math. Phys.}, 342:441--489, 2016.

\bibitem{boscheri2021structure}
Walter Boscheri, Michael Dumbser, Matteo Ioriatti, Ilya Peshkov, and Evgeniy
  Romenski.
\newblock A structure-preserving staggered semi-implicit finite volume scheme
  for continuum mechanics.
\newblock {\em Journal of Computational Physics}, 424:109866, 2021.

\bibitem{bul92}
A.V. Bulinskii.
\newblock {Central Limit Theorem for the Solution of the Multidimensional
  Burgers Equation with Random Data}.
\newblock {\em Annales Academie Scientiarum Fennice}, 17:11--22, 1992.

\bibitem{cer01}
V.~Cerveny.
\newblock {\em Seismic Ray Theory}.
\newblock Cambridge University Press, 2001.

\bibitem{chaconvladimirsky2012fast}
Adam Chacon and Alexander Vladimirsky.
\newblock Fast two-scale methods for eikonal equations.
\newblock {\em SIAM Journal on Scientific Computing}, 34(2):A547--A578, 2012.

\bibitem{chechokao13}
W.~Chen, C.S. Chou, and C.Y. Kao.
\newblock {L}ax-{F}riedrichs fast sweeping methods for steady state problems
  for hyperbolic conservation laws.
\newblock {\em J.Comput.Phys.}, 234:452--471, 2013.

\bibitem{cor83}
S.~Cornbleet.
\newblock Geometrical optics reviewed: A new light on an old subject.
\newblock {\em {Proc.} {IEEE}}, 71:471--502, 1983.

\bibitem{cristianifalcone2007fast}
Emiliano Cristiani and Maurizio Falcone.
\newblock Fast semi-lagrangian schemes for the eikonal equation and
  applications.
\newblock {\em SIAM Journal on Numerical Analysis}, 45(5):1979--2011, 2007.

\bibitem{fonleu23}
K.W. Fong and S.~Leung.
\newblock {Spherical Essentially Non-Oscillatory (SENO) Interpolation}.
\newblock {\em J. Sci. Comput.}, 94(28), 2023.

\bibitem{ham63}
S.W.R. Hamilton.
\newblock {\em Elements of Quaternions}.
\newblock Chelsea Publishing Co., London, 1963.

\bibitem{hqcwwl21}
J.~Hu, J.~Qian, J.~Cao, X.~Wang, H.~Wang, and S.~Leung.
\newblock Ray-illumination compensation for adjoint-state first-arrival
  traveltime tomography.
\newblock {\em {Geophysics}}, 86(5):U109--U119, 2021.

\bibitem{huzha24}
R.~Hu and Y.~T. Zhang.
\newblock {High Order Absolutely Convergent Fast Sweeping Methods with
  Multi-resolution WENO Local Solvers for Eikonal and Factored Eikonal
  Equations}.
\newblock {\em J. Sci. Comput.}, 99(67), 2024.

\bibitem{iserlesetal2000lie}
Arieh Iserles, Hans~Z. Munthe-Kaas, Syvert~P. N{\o}rsett, and Antonella Zanna.
\newblock Lie-group methods.
\newblock {\em Acta Numerica}, 9:215--365, 2000.

\bibitem{jeongwhitaker2008fast}
Won-Ki Jeong and Ross~T. Whitaker.
\newblock A fast iterative method for eikonal equations.
\newblock {\em SIAM Journal on Scientific Computing}, 30(5):2512--2534, 2008.

\bibitem{julgub77}
B.R. Julian and D.~Gubbins.
\newblock Three-dimensional seismic ray tracing.
\newblock {\em J.Geophys.}, 43:95--113, 1977.

\bibitem{kaooshqia04}
C.Y. Kao, S.J. Osher, and J.~Qian.
\newblock {L}ax-{F}riedrichs sweeping schemes for static {H}amilton-{J}acobi
  equations.
\newblock {\em J. Comp. Phys.}, 196:367--391, 2004.

\bibitem{lar25}
A.~Larios.
\newblock Burgers equation with a twist: A study on rotational-form equations.
\newblock {\em arXiv:2510.02761}, 2025.

\bibitem{leeleu26}
Y.K. Lee and S.~Leung.
\newblock {High-Order Spherical Runge-Kutta Schemes Based on Spherical Linear
  Extrapolation}.
\newblock {\em J. Sci. Comput.}, 107(3):103, 2026.

\bibitem{leuchalee24}
S.~Leung, W.M. Chau, and Y.K. Lee.
\newblock {SLERP-TVDRK (STVDRK) Methods for Ordinary Differential Equations on
  Spheres}.
\newblock {\em J. Sci. Comput. (arXiv:2410.10420)}, 101(63), 2024.

\bibitem{leuqia0601}
S.~Leung and J.~Qian.
\newblock An adjoint state method for three-dimensional transmission traveltime
  tomography using first-arrivals.
\newblock {\em Commun.Math.Sci.}, 4:249--266, 2006.

\bibitem{lileuqia14}
W.B. Li, S.~Leung, and J.~Qian.
\newblock A level-set adjoint-state method for crosswell
  transmission-reflection traveltime tomography.
\newblock {\em Geophs. J. Int.}, 199(1):348--367, 2014.

\bibitem{luoqiabur14}
S.~Luo, J.~Qian, and R.~Burridge.
\newblock High-order factorization based high-order hybrid fast sweeping
  methods for point-source eikonal equations.
\newblock {\em SIAMJ.Num.Anal.}, 52:23--44, 2014.

\bibitem{miksiszhang2024sparse}
Zachary~M. Miksis and Yong-Tao Zhang.
\newblock Sparse-grid implementation of fixed-point fast sweeping {WENO}
  schemes for eikonal equations.
\newblock {\em Communications on Applied Mathematics and Computation},
  6(1):3--29, 2024.
\newblock First published online in 2022.

\bibitem{ockst02}
S.~Osher, L.-T. Cheng, M.~Kang, H.~Shim, and Y-H Tsai.
\newblock {Geometric Optics in a Phase-Space-Based Level Set and Eulerian
  Framework}.
\newblock {\em J. Comput. Phys.}, 179(2):622--648, 2002.

\bibitem{Oshfed03}
S.~J. Osher and R.~P. Fedkiw.
\newblock {\em Level Set Methods and Dynamic Implicit Surfaces}.
\newblock Springer-Verlag, New York, 2003.

\bibitem{oshset88}
S.~J. Osher and J.~A. Sethian.
\newblock Fronts propagating with curvature dependent speed: algorithms based
  on {Hamilton-Jacobi} formulations.
\newblock {\em J.Comput.Phys.}, 79:12--49, 1988.

\bibitem{ple06}
R.E. Plessix.
\newblock A review of the adjoint-state method for computing the gradient of a
  functional with geophysical applications.
\newblock {\em Geophys.J.Internat.}, 167:495--503, 2006.

\bibitem{popset97}
A.~M. Popovici and J.~A. Sethian.
\newblock Three-dimensional traveltime computation using the fast marching
  method.
\newblock In {\em 67th Ann. Internat. Mtg., Soc. Expl. Geophys., Expanded
  Abstracts}, pages 1778--1781. Soc. Expl. Geophys., 1997.

\bibitem{qialeu04}
J.~Qian and S.~Leung.
\newblock A level set based {E}ulerian method for paraxial multivalued
  traveltimes.
\newblock {\em J.Comput.Phys.}, 197:711--736, 2004.

\bibitem{qiazhazha0702}
J.~Qian, Y-T Zhang, and H-K. Zhao.
\newblock A fast sweeping method for static convex {H}amilton--{J}acobi
  equations.
\newblock {\em Journal of Scientific Computing}, 31(1--2):237--271, 2007.

\bibitem{qiazhazha07}
J.~Qian, Y-T Zhang, and H-K. Zhao.
\newblock Fast sweeping methods for eikonal equations on triangulated meshes.
\newblock {\em SIAM J. Numer. Analy.}, 45:83--107, 2007.

\bibitem{rawhausam07}
N.~Rawlinson, J.~Hauser, and M.~Sambridge.
\newblock Seismic ray tracing and wavefront tracking in laterally heterogeneous
  media.
\newblock {\em Adv. Geophs.}, 49:203--267, 2007.

\bibitem{sander2012geodesic}
Oliver Sander.
\newblock Geodesic finite elements on simplicial grids.
\newblock {\em International Journal for Numerical Methods in Engineering},
  92(12):999--1025, 2012.

\bibitem{sander2016higher}
Oliver Sander.
\newblock Geodesic finite elements of higher order.
\newblock {\em IMA Journal of Numerical Analysis}, 36(1):238--266, 2016.

\bibitem{Sethian1996}
J~A Sethian.
\newblock A fast marching level set method for monotonically advancing fronts.
\newblock {\em Proceedings of the National Academy of Sciences},
  93(4):1591--1595, February 1996.

\bibitem{set96}
J.~A. Sethian.
\newblock {\em Level set methods}.
\newblock Cambridge Univ. Press, 1996.

\bibitem{set99siamrev}
J.~A. Sethian.
\newblock Fast marching methods.
\newblock {\em SIAMReview}, 41(2):199--235, 1999.

\bibitem{shoemake_85}
K.~Shoemake.
\newblock Animating rotation with quaternion curves.
\newblock In {\em Proceedings of the 12th annual conference on Computer
  graphics and interactive techniques}, pages 245--254, 1985.

\bibitem{sola}
J.~Sol{\`{a}}.
\newblock Quaternion kinematics for the error-state {K}alman filter.
\newblock {\em arXiv:1711.02508 [CS.RO]}, 2017.

\bibitem{sym98}
W.~W. Symes.
\newblock A slowness matching finite difference method for traveltimes beyond
  transmission caustics.
\newblock In {\em 68th Ann. Internat. Mtg., Soc. Expl. Geophys., Expanded
  Abstracts}, pages 1945--1948. Soc. Expl. Geophys., 1998.

\bibitem{symqia03slowm}
W.~W. Symes and J.~Qian.
\newblock A slowness matching {E}ulerian method for multivalued solutions of
  eikonal equations.
\newblock {\em J. Sci. Comp.}, 19:501--526, 2003.

\bibitem{tncc09}
C.~Taillandier, M.~Noble, H.~Chauris, and H.~Calandra.
\newblock First-arrival traveltime tomography based on the adjoint-state
  method.
\newblock {\em Geophysics}, 74:WCB1--WCB10, 2009.

\bibitem{tcoz03}
R.~Tsai, L.T. Cheng, S.~Osher, and H.~K. Zhao.
\newblock Fast sweeping method for a class of {H}amilton-{J}acobi equations.
\newblock {\em SIAM J. Numer. Anal.}, 41:673--694, 2003.

\bibitem{tsi95}
J.N. Tsitsiklis.
\newblock Efficient algorithms for globally optimal trajectories.
\newblock {\em IEEE Tran. Automatic Control}, 40:1528--1538, 1995.

\bibitem{viag96}
V.~Vinje, E.~Iversen, K.~{\AA}steb{\o}l, and H.~Gj{\o}ystdal.
\newblock Estimation of multivalued arrivals in 3d models using wavefront
  construction - part 1.
\newblock {\em Geophys.Prosp.}, 44:819--842, 1996.

\bibitem{wonleu15}
T.~Wong and S.~Leung.
\newblock A fast sweeping method for eikonal equations on implicit surfaces.
\newblock {\em Accepted by J. Sci. Comput.}, 2015.

\bibitem{xuzhayon10}
S.G. Xu, Y.X. Zhang, and J.H. Yong.
\newblock A fast sweeping method for computing geodesics on triangular
  manifolds.
\newblock {\em IEEE Transactions on Pattern Analysis and Machine Intelligence},
  32(2):231--241, 2010.

\bibitem{zhashu03}
Y.-T. Zhang and C.-W. Shu.
\newblock High order {WENO} schemes for {H}amilton-{J}acobi equations on
  triangular meshes.
\newblock {\em SIAM J. Sci. Comp.}, 24:1005--1030, 2003.

\bibitem{zclzs11}
Y.T. Zhang, S.~Chen, F.~Li, H.~K. Zhao, and C.~W. Shu.
\newblock {Uniformly Accurate Discontinuous Galerkin Fast Sweeping Methods for
  Eikonal Equations}.
\newblock {\em {SIAM J. Sci. Comput.}}, 33(4):1873--1896, 2011.

\bibitem{zhazhaqia06}
Y.T. Zhang, H.K. Zhao, and J.~Qian.
\newblock {High Order Fast Sweeping Methods for Static Hamilton--Jacobi
  Equations}.
\newblock {\em J.Sci.Comput.}, 29:25--56, 2006.

\bibitem{zha05}
H.~K. Zhao.
\newblock Fast sweeping method for eikonal equations.
\newblock {\em Math. Comp.}, 74:603--627, 2005.

\end{thebibliography}

\end{document}